\RequirePackage{fix-cm}
\documentclass[smallextended]{svjour3}       
\smartqed  
\usepackage{graphicx}
\usepackage{amssymb}
\usepackage{amsmath}
\usepackage{mathtools}
\usepackage{cancel}
\usepackage{xcolor}
\usepackage{colortbl}
\usepackage{MnSymbol}
\usepackage[ruled,vlined]{algorithm2e}
\usepackage{ulem}
\usepackage{caption}
\usepackage{subcaption}
\usepackage{mathrsfs}
\usepackage{float}
\usepackage[usestackEOL]{stackengine}
\edef\tmp{\the\baselineskip}
\setstackgap{L}{\tmp}
\DeclareMathAlphabet{\mathpzc}{OT1}{pzc}{m}{it}
\definecolor{Gray}{gray}{0.95}
\definecolor{Gray2}{gray}{0.99}

\newcommand{\mesh}{\mathcal{M}}

\newcommand{\mup}{\boldsymbol{\mu}}

\newcommand{\vb}{\textbf{v}}

\newcommand{\x}{\mathbf{x}}
\newcommand{\y}{\mathbf{y}}

\newcommand{\s}{\mathbf{s}}
\newcommand{\geodistance}{\mathpzc{d}}

\newcommand{\dist}{\textnormal{dist}}

\newcommand{\proj}{\Pi}

\newcommand{\fomdim}{N_{h}}

\newcommand{\weight}{\mathbf{W}}
\newcommand{\bias}{\mathbf{b}}

\newcommand{\diam}{\textnormal{diam}}
\newcommand{\meshtomesh}[2]{\triplet{V_{#2}}{\cong}{\mathbb{R}^{#1}}}
\newcommand{\triplet}[3]{\begin{tabular}{l}\\\\$#1$  \\$#2$ \\$#3$\end{tabular}}
\newcommand{\hcoarse}{4h}
\newcommand{\hintermediate}{2h}
\newcommand{\hfine}{h}
\newcommand{\newstuff}[1]{#1}
\newcommand{\acapo}{\newline\indent}

\begin{document}


\title{\newstuff{Mesh-Informed Neural Networks for Operator Learning in Finite Element Spaces}}


\author{Nicola Rares Franco\and Andrea Manzoni\and Paolo Zunino}


\institute{Nicola Rares Franco \at
              MOX, Department of Mathematics, Politecnico di Milano, Italy \\
              \email{nicolarares.franco@polimi.it}
           \and
           Andrea Manzoni \at
              MOX, Department of Mathematics, Politecnico di Milano, Italy \\
              \email{andrea1.manzoni@polimi.it}
           \and
           Paolo Zunino \at
              MOX, Department of Mathematics, Politecnico di Milano, Italy \\
              \email{paolo.zunino@polimi.it}
}

\date{}

\maketitle

\vspace{-2em}
\begin{abstract} 
Thanks to their universal approximation properties and new efficient training strategies, Deep Neural Networks are becoming a valuable tool for the approximation of mathematical operators. In the present work, we introduce Mesh-Informed Neural Networks (MINNs), a class of architectures specifically tailored to handle mesh based functional data, and thus of particular interest for reduced order modeling of parametrized Partial Differential Equations (PDEs). The driving idea behind MINNs is to embed hidden layers into discrete functional spaces of increasing complexity, obtained through a sequence of meshes defined over the underlying spatial domain. The approach leads to a natural pruning strategy which enables the design of sparse architectures that are able to learn general nonlinear operators. We assess this strategy through an extensive set of numerical experiments, ranging from nonlocal operators to nonlinear diffusion PDEs, where MINNs are compared \newstuff{against more traditional architectures, such as} classical fully connected Deep Neural Networks\newstuff{, but also more recent ones, such as DeepONets and Fourier Neural Operators}. Our results show that MINNs can handle functional data defined on general domains of any shape, while ensuring reduced training times, lower computational costs, and better generalization capabilities, thus making MINNs very well-suited for demanding applications such as Reduced Order Modeling and Uncertainty Quantification for PDEs.

\end{abstract}

\section{Introduction}
\newcommand{\operator}{\mathcal{G}}
\label{sec:intro}
Deep Neural Networks (DNNs) are one of the fundamental building blocks in modern Machine Learning. Originally developed to tackle classification tasks, they have become extremely popular after reporting striking achievements in fields such as computer vision \cite{vision} and language processing \cite{language}. Not only, an in-depth investigation of their approximation properties has also been carried out in the last decade \cite{brener,devore,francocnn,gribonval,ingo}.  In particular, DNNs have been recently employed for learning (nonlinear) operators in high-dimensional spaces \cite{schwabhigh,geist,kovachki,karniadakis}, because of their unique properties, such as the ability to blend theoretical and data-driven approaches. Additionally, the interest in using DNNs to learn high-dimensional operators arises from the potential repercussions that these models would have on fields such as Reduced Order Modeling.

Consider for instance a parameter dependent PDE problem, where each parameter instance $\mup$ leads to a solution $u_{\mup}$. In this framework, multi-query applications such as optimal control and statistical inference are prohibitive to implement, \newstuff{as they imply repeated queries to expensive numerical solvers.} \newstuff{Then}, learning the operator $\mup\to u_{\mup}$ \newstuff{becomes} of key interest, as it \newstuff{allows} 
one to replace numerical solvers with much cheaper surrogates. To this end, DNNs can be a valid and powerful alternative, \newstuff{as they were recently shown capable of} either comparable or superior results with respect to other state-of-the-art, e.g. \cite{stuart,fresca,fresca2}. 
\newstuff{More generally}, other works have recently exploited physics-informed machine learning for efficient reduced order modeling of parametrized PDEs \cite{ChenPINN,paris}. 
Also, DNN models have the practical advantage of being highly versatile as, differently from other techniques such as splines and wavelets, they can easily adapt to both high-dimensional inputs, as in image recognition, and outputs, as in the so-called generative models.

However, when the dimensions into play become very high, there are some practical issues that hinder the use of as-is DNN models. In fact, classical dense architectures tend to have too many degrees of freedom, \newstuff{which makes} them harder to train, computationally demanding and prone to overfitting \cite{augasta}. As a remedy, alternative architectures such as Convolutional Neural Networks (CNNs) and Graph Neural Networks (GNNs) have been employed over the years. These architectures can handle very efficiently data defined respectively over hypercubes (CNNs) or graphs (GNNs). Nevertheless, these models do not provide a complete answer, especially when the high-dimensionality arises from the discretization of a functional space such as $L^{2}(\Omega)$, where $\Omega\subset\mathbb{R}^{d}$ is some bounded domain, possibly nonconvex. In fact, CNNs cannot handle general geometries and they might become inappropriate as soon as $\Omega$ is not an hypercube, although some preliminary attempts to generalize CNN in this direction have recently appeared \cite{Gao}.
Conversely, GNNs have the benefit of considering their inputs and outputs as defined over the vertices of a graph \cite{scarselli}. This appears to be a promising feature, since a classical way to discretize spatial domains is to use meshing strategies, and meshes are ultimately graphs. However, GNNs are heavily based on the graph representation itself, and their construction does not exploit the existence of an underlying spatial domain. In particular, GNNs do not work at different  levels of \newstuff{resolution}, which would be a desirable property \newstuff{in a context in which the discretization is, ultimately, fictitious.}
\\\\
Inspired by these considerations, we introduce a novel class of sparse architectures, which we refer to as Mesh-Informed Neural Networks (MINNs), to tackle the problem of learning a (nonlinear) operator $$\newstuff{\operator:V_{1}\to V_{2},}$$
where \newstuff{$V_{1}$ and $V_{2}$} are some functional spaces, e.g. $\newstuff{V_{1}=V_{2}}\subseteq L^{2}(\Omega)$. The definition of $\operator$ may involve both local and nonlocal operations, such as derivatives and integrals, and it may as well imply the solution to a Partial Differential Equation (PDE). 

We cast the above problem in a setting that is more familiar to the Deep Learning literature by introducing a form of high-fidelity discretization. This is an approach that has become widely adopted by now, and it involves the discretization of the functional spaces along the same lines of Finite Element methods, as in \cite{stuart,fresca,kutyniok}. In short, one introduces a mesh having vertices $\{\x_{i}\}_{i=1}^{\fomdim}\subset\overline{\Omega}$, and defines $V_{h}\subset L^{2}(\Omega)$ as the subspace of piecewise-linear Lagrange polynomials, where $h>0$ is the stepsize of the mesh (the idea can be easily generalized to higher order finite elements, as we show later on). Since each $v\in V_{h}$ is uniquely identified by its nodal values, we have $V_{h}\cong\mathbb{R}^{\fomdim}$, and the original operator to be learned can be replaced by $$\operator_{h}:V_{h}\cong\mathbb{R}^{\fomdim}\to V_{h}\cong\mathbb{R}^{\fomdim}.$$
The idea is now to approximate $\operator_{h}$ by training some DNN $\Phi:\mathbb{R}^{\fomdim}\to\mathbb{R}^{\fomdim}$. As we argued previously, dense architectures are unsuited for such a purpose because of their prohibitive computational cost during training, \newstuff{which is mostly caused by: i) the computational resources required for the optimization, ii) the amount of training data needed to avoid overfitting.}

To overcome this bottleneck, we propose Mesh-Informed Neural Networks. These are ultimately based on an \textit{apriori} pruning strategy, whose purpose is to inform the model with the geometrical knowledge coming from $\Omega$. As we will demonstrate later in the paper, despite their simple implementation, MINNs show reduced training times and better generalization capabilities, making them \newstuff{a competitive alternative to other operator learning approaches, such as DeepONets \cite{karniadakis} and Fourier Neural Operators \cite{fnos}.}
Also, they allow for a novel interpretation of the so-called \textit{hidden layers}, in a way that may simplify the practical problem of designing DNN architectures. 
\\\\
The rest of the paper is devoted to the presentation of MINNs and it is organized as follows. In Section \ref{sec:minns}, we set some notation and formally introduce Mesh-Informed Neural Networks from a theoretical point of view. There, we also discuss their implementation and comment on the parallelism between MINNs and other emerging approaches such as DeepONets \cite{karniadakis} and Neural Operators \cite{kovachki}. \newstuff{We then devote Sections \ref{sec:experiments1}, \ref{sec:experiments2} and \ref{sec:experiments4} to the numerical experiments, addressing a different scientific question in each Section. More precisely: in Section \ref{sec:experiments1}, we provide empirical evidence that the pruning strategy underlying MINNs is powerful enough to resolve the issues of dense architectures; in Section \ref{sec:experiments2}, we showcase the flexibility of MINNs in handling complex nonconvex domains; finally, in Section \ref{sec:experiments4}, we compare the performances of MINNs with those of other state-of-the-art Deep Learning algorithms, namely DeepONets and Fourier Neural Operators.
Following the numerical experiments, in Section \ref{sec:uq},} we take the chance to present an application where MINNs are employed to answer a practical problem of Uncertainty Quantification related to the delivery of oxygen in biological tissues. Finally, we draw our conclusions and discuss future developments in Section \ref{sec:conclusions}.

\section{Mesh-Informed Neural Networks}
\label{sec:minns}
In the present Section we present Mesh-Informed Neural Networks, a novel class of architectures specifically built to handle \newstuff{discretized functional data defined over meshes}, and thus of particular interest \newstuff{for} PDE applications. Preliminary to that, we introduce some notation and recall some of the basic concepts behind classical DNNs.

\subsection{Notation and preliminaries}
Deep Neural Networks are a powerful class of approximators that is ultimately based on the composition of affine and nonlinear transformations. Here, we focus on DNNs having a \textit{feedforward} architecture. We report below some basic definitions.

\begin{definition}
\label{def:layer}
Let $m,n\ge1$ and $\rho:\mathbb{R}\to\mathbb{R}$. A layer with activation $\rho$ is a map $L:\mathbb{R}^{m}\to\mathbb{R}^{n}$ of the form $L(\vb)=\rho\left(\weight\vb+\bias\right)$, for some $\weight\in\mathbb{R}^{n\times m}$ and $\bias\in\mathbb{R}^{n}$.
\end{definition}

In the literature, $\weight$ and $\bias$ are usually referred to as weight and bias of the layer, respectively. Note that the definition above contains an abuse of notation, as $\rho$ is evaluated over an $n$-dimensional vector: we understand the latter operation component-wise, that is $\rho([x_{1},\ldots,x_{n}]):=[\rho(x_{1}),\ldots,\rho(x_{n})]$.

\begin{definition}
Let $m,n\ge1$. A neural network of depth $l\ge0$ is a map $\Phi:\mathbb{R}^{m}\to\mathbb{R}^{n}$ obtained via composition of $l+1$ layers, $\Phi=L_{l+1}\circ\ldots L_{1}$.
\end{definition}

The layers of a neural network do not need to share the same activation function and usually the output layer, $L_{l+1}$, does not have one. Architectures with $l=1$ are known as shallow networks, while the adjective deep is used when $l\ge2$. We also allow for the degenerate case in which the network reduces to a single layer ($l=0$). The classical pipeline for building a neural network model starts by fixing the architecture, that is the number of layers and their input-output dimensions. Then, the weights and biases of all layers are tuned accordingly to some procedure, which typically involves the optimization of a loss function computed over a given training set.

\subsection{Mesh-Informed layers}

We consider the following framework. We are given a bounded domain $\Omega\subset\mathbb{R}^{d}$, not necessarily convex, and two meshes \newstuff{(see Definition \ref{def:meshes} right below)} having respectively stepsizes $h,h'>0$ and vertices $$\{\x_{j}\}_{j=1}^{N_{h}},\;\{\x_{i}'\}_{i=1}^{N_{h'}}\subset\overline{\Omega}.$$ The two meshes can be completely different and they can be either structured or unstructured. To each mesh we associate the corresponding space of piecewise-linear Lagrange polynomials, namely $V_{h},V_{h'}\subset L^{2}(\Omega)$. Our purpose is to introduce a suitable notion of \textit{mesh-informed layer} $L: V_{h}\to V_{h'}$ that exploits the apriori existence of $\Omega$. In analogy to Definition \ref{def:layer}, $L$ should have $N_{h}$ neurons at input and $N_{h'}$ neurons at output, since $V_{h}\cong\mathbb{R}^{N_{h}}$ and $V_{h'}\cong\mathbb{R}^{N_{h'}}$. However, thinking of the state spaces as either comprised of functions or vectors is fundamentally different: while we can describe the objects in $V_{h}$ as regular, smooth or noisy, these notions have no meaning in $\mathbb{R}^{\fomdim}$, and similarly for $V_{h'}$ and $\mathbb{R}^{N_{h'}}$. Furthermore, in the case of PDE applications, we are typically not interested in all the elements of $V_{h}$ and $V_{h'}$, rather we focus on those that present spatial correlations coherent with the underlying physics. Starting from these considerations, we build a novel layer architecture that can meet our specific needs. In order to provide a rigorous definition, and directly extend the idea to higher order finite element spaces, we first introduce some preliminary notation. For the sake of simplicity, we will restrict to simplicial finite elements \cite{simplicial}.

\begin{definition}
\label{def:meshes}
Let $\Omega\subset\mathbb{R}^{d}$ be a bounded domain. Let $\mesh$ 
be a collection of $d$-simplices in $\Omega$, \newstuff{so that each $K\in\mesh$ is a closed subset of $\overline{\Omega}$}. For each \textit{element} $K\in\mesh$, define the quantities
$$h_{K}:=\diam(K),\quad\quad R_{K}:=\sup\;\left\{\diam(S)\;|\;S\text{ is a ball contained in }K\right\}.$$
We say that $\mesh$ is an admissible mesh of stepsize $h>0$ over $\Omega$ if the following conditions hold.
\begin{itemize}
    \item [1.] The elements are exhaustive, that is
    $$\dist\left(\overline{\Omega}, \bigcup_{K\in\mesh}K\right)\le h$$
    where $\dist(A,B)=\sup_{x\in A}\inf_{y\in B}|x-y|$ is the distance between $A$ and $B$.\\
    
    \item [2.] Any two distinct elements $K,K'\in\mesh$ have disjoint interiors. Also, their intersection is either empty or results in a common face of dimension $s<d$.\\

    \item [4.] The elements are non degenerate and their maximum diameter equals $h$, that is
    $$\min_{K\in\mesh}\;R_{K}>0\quad\text{and}\quad\max_{K\in\mesh}\;h_{K}=h.$$
\end{itemize}
In that case, the quantity
$$\sigma = \min_{K\in\mesh}\;\frac{h_{K}}{R_{K}}<+\infty,$$
is said to be the aspect-ratio of the mesh.
\end{definition}

\begin{definition}
\label{def:fe}
Let $\Omega\subset\mathbb{R}^{d}$ be a bounded domain and let $\mesh$ 
be a mesh of stepsize $h>0$ defined over $\Omega$. For any positive integer $q$, we write $X_{h}^{q}(\mesh)$ for the finite element space of piecewise-polynomials of degree at most $q$, that is
$$X_{h}^{q}(\mesh):=\{v\in\mathcal{C}(\overline{\Omega})\;\text{s.t.}\;v_{|K}\;\text{is a polynomial of degree at most }q\;\;\forall K\in\mesh\}.$$
Let $\fomdim=\dim(X_{h}^{q}(\mesh))$. We say that a collection of nodes $\{\x_{i}\}_{i=1}^{\fomdim}\subset\Omega$ and a sequence of functions $\{\varphi_{i}\}_{i=1}^{\fomdim}\subset X_{h}^{q}(\mesh)$ define a Lagrangian basis of $X_{h}^{q}(\mesh)$ if
$$\varphi_{j}(\x_{i})=\delta_{i,j}\quad\quad i,j=1,\dots,N_{h}.$$
We write $\proj_{h,q}(\mesh):X_{h}^{q}(\mesh)\to \mathbb{R}^{\fomdim}$ for the function-to-nodes operator,
$$\proj_{h,q}(\mesh):\;\;v \to [v(\x_{1}),\dots,v(\x_{\fomdim})],$$
whose inverse is
$$\proj_{h,q}^{-1}(\mesh):\;\;\mathbf{c} \to \sum_{i=1}^{\fomdim}c_{i}\varphi_{i}.$$
\end{definition}

We now have all we need to introduce our concept of mesh-informed layer.

\begin{definition}
\label{def:meshinformed} \textnormal{(Mesh-informed layer)}
Let $\Omega\subset\mathbb{R}^{d}$ be a bounded domain and \newstuff{$\geodistance:\Omega\times\Omega\to[0,+\infty)$ a given distance function}. Let $\mesh$ and $\mesh'$ be two meshes of stepsizes $h$ and $h'$, respectively. Let $V_{h}=X_{h}^{q}(\mesh)$ and $V_{h'}=X_{h'}^{q'}(\mesh)$ be the input and output spaces, respectively. Denote by $\{\x_{j}\}_{j=1}^{N_{h}}$ and $\{\x_{i}'\}_{i=1}^{N_{h'}}$ the nodes associated to a Lagrangian basis of $V_{h}$ and $V_{h'}$ respectively. A mesh-informed layer with activation function $\rho:\mathbb{R}\to\mathbb{R}$ and support $r>0$ is a map $L:V_{h}\to V_{h'}$ of the form
\begin{equation*}L = \proj_{h',q'}^{-1}(\mesh')\circ \tilde{L} \circ \proj_{h,q}(\mesh)\end{equation*}
where $\tilde{L}:\mathbb{R}^{\fomdim}\to\mathbb{R}^{N_{h}'}$ is a layer with activation $\rho$ whose weight matrix $\weight$ satisfies the additional sparsity constraint below,
\begin{equation*}\newstuff{\geodistance(\x_{j},\x_{i}')>r} \implies \weight_{i,j} = 0.\end{equation*}
\end{definition}

\begin{figure}
\centering
\includegraphics[width=0.75\textwidth]{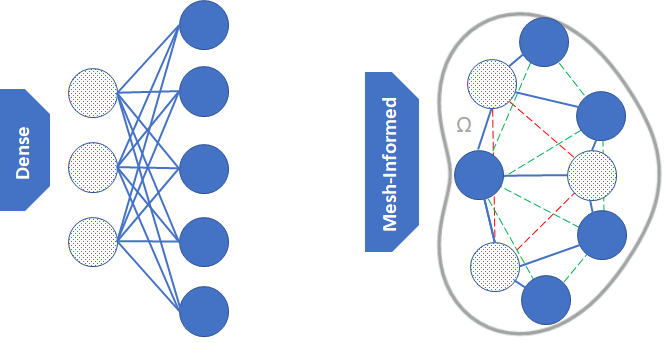}
\caption{\label{fig:meshinfo}Comparison of a dense layer (cf. Definition 1) and a mesh-informed layer (cf. Definition 5). The dense model features 3 neurons at input and 5 at output. All the neurons communicate: consequently the weight matrix of the layer has 15 active entries. In the mesh-informed counterpart, neurons become vertices of two meshes (resp. in green and red) defined over the same spatial domain $\Omega$. Only nearby neurons are allowed to communicate. This results in a sparse model with only 9 active weights.}
\end{figure}

\newstuff{The distance function $\geodistance$ in Definition \ref{def:meshinformed} can be any metric over $\Omega$. For instance, one may choose to consider the Euclidean distance, $\geodistance(\x,\x'):=|\x-\x'|$. However, if the geometry of $\Omega$ becomes particularly involved, better choices of $\geodistance$ might be available, such as the \textit{geodesic distance}. The latter quantifies the distance of two points $\x,\x'\in\Omega$ by measuring the length of the shortest path within $\Omega$ between $\x$ and $\x'$, namely
\begin{multline*}
    \geodistance(\x,\x'):=\inf\;\bigg\{\int_{0}^{1}|\gamma'(t)|dt,\;\text{with}\;  \gamma\in\mathcal{C}([0,1],\mathbb{R}^{d}),\;\gamma([0,1])\subseteq\Omega,\\\gamma(0)=\x,\;\gamma(1)=\x'\bigg\}.\;\;\;
\end{multline*}
}

It is worth pointing out that, as a matter of fact, the projections $\proj_{h,q}(\mesh)$ and $\proj_{h',q'}^{-1}(\mesh')$ have the sole purpose of making Definition \ref{def:meshinformed} rigorous. What actually defines the mesh-informed layer $L$ are the sparsity patterns imposed to $\tilde{L}$.
In fact, the idea is that these constraints should help the layer in producing outputs that are more coherent with the underlying spatial domain (cf. Figure \ref{fig:meshinfo}). In light of this intrinsic duality between $L$ and $\tilde{L}$, we will refer to the weights and biases of $L$ as to those that are formally defined in $\tilde{L}$. Also, for better readability, from now on we will use the notation \begin{equation*}L: V_{h}\xrightarrow[]{\;r\;} V_{h'},\end{equation*} to emphasize that $L$ is a mesh-informed layer with support $r$. We note that dense layers can be recovered by letting \newstuff{$r\ge\sup\;\{\geodistance(\x,\x')\;|\;\x,\x'\in\Omega\}$}, while lighter architecture are obtained for smaller values of $r$. The following result provides an explicit upper bound on the number of nonzero entries in a mesh-informed layer. \newstuff{For the sake of simplicity, we restrict to the case in which $\geodistance$ is the Euclidean distance.}

\begin{theorem}\label{theorem:sparsity}
\textit{Let $\Omega\subset\mathbb{R}^{d}$ be a bounded domain \newstuff{and $\geodistance$ the Euclidean distance}. Let $\mesh$ and $\mesh'$ be two meshes having respectively stepsizes $h,h'$ and aspect-ratios $\sigma,\sigma'$. Let \begin{equation*}h_{\min}:=\min_{K\in\mesh}h_{K},\quad h'_{\min}:=\min_{K'\in\mesh'}h_{K'}\end{equation*} be the smallest diameters within the two meshes respectively. Let $L: V_{h}\xrightarrow[]{\;r\;} V_{h'}$ be a mesh-informed layer of support $r>0$, where $V_{h}:=X_{h}^{q}(\mesh)$ and $V_{h'}:=X_{h'}^{q'}(\mesh')$. Then,
\begin{equation*}\|\mathbf{W}\|_{0} \le C\left(\sigma\sigma'\frac{r}{h_{\min}h'_{\min}}\right)^{d}\end{equation*}
where $\|\mathbf{W}\|_{0}$ is the number of nonzero entries in the weight matrix of the layer $L$, while $C=C(\Omega,d,q,q')>0$ is a constant depending only on $\Omega$, $d$, $q$ and $q'$.}
\end{theorem}
\begin{proof}
Let $N_{h}:=\text{dim}(V_{h})$, $N_{h'}:=\text{dim}(V_{h'})$ and let $\{\x_{j}\}_{j=1}^{N_{h}},\{\x_{i}'\}_{i=1}^{N_{h'}}$ be the Lagrangian nodes in the two meshes respectively. Let $\omega:=|B(\mathbf{0},1)|$ be the volume of the unit ball in $\mathbb{R}^{d}$. Since $\min_{K'}R_{K'}\ge h'_{\min}/\sigma'$, the volume of an element in the output mesh is at least
$$v_{\min}(h'):= (h'_{\min}/\sigma')^{d}\omega.$$
It follows that, for any $\x\in\Omega$, the ball $B(\x,r)$ can contain at most 
$$n_{\text{e}}(r,h'):=\frac{\omega r^{d}}{v_{\min}(h')}=\left(\frac{\sigma' r}{h'_{\min}}\right)^{d}$$
elements of the output mesh. Therefore, the number of indices $i$ such that $|\x_{i}'-\x_{j}|\le r$ is at most $n_{\text{e}}(r,h')c(d,q')$, where $c(d,q'):=(d+q')!/(q'!d!)$ bounds the number of degrees of freedom within each element. Finally,
\begin{multline*}
\|\weight\|_{0}\le N_{h}n_{\text{e}}(r,h')c(d,q')\le \\ \le \frac{c(d,q)|\Omega|}{v_{\min}(h)}n_{\text{e}}(r,h')c(d,q')=\frac{c(d,q)c(d,q')|\Omega|}{\omega}\cdot\left(\frac{\sigma\sigma' r}{h_{\min}h'_{\min}}\right)^{d}    
\end{multline*}
\end{proof}

Starting from here, we define MINNs  by composition, with a possible interchange of dense and mesh-informed layers. Consider for instance the case in which we want to define a Mesh-Informed Neural Network $\Phi: \mathbb{R}^{p}\to V_{h}\cong\mathbb{R}^{\fomdim}$ that maps a low-dimensional input, say $p\ll N_{h}$, to some functional output. Then, using our notation, one possible architecture could be\vspace{-2em}
\begin{equation}
    \label{eq:arch}
    \Phi:\mathbb{R}^{p}\xrightarrow[]{\;\;\;\;\;\;\;\;\;}
    \meshtomesh{N_{\hcoarse}}{\hcoarse}
    \xrightarrow[]{r\;=\;0.5\;}
    \meshtomesh{N_{\hintermediate}}{\hintermediate}
    \xrightarrow[]{r\;=\;0.25}
    \meshtomesh{N_{\hfine}}{\hfine},
\end{equation}
The above scheme defines a DNN of depth $l=2$, as it is composed of $3$ layers. The first layer is dense (Definition \ref{def:layer}) and has the purpose of preprocessing the input while mapping the data onto a coarse mesh (stepsize $4h$). Then, the remaining two layers perform local transformations in order to return the desired output. Note that the three meshes need not to satisfy any hierarchy, see e.g. Figure \ref{fig:meshes}. Also, the corresponding finite element spaces need not to share the same polynomial degree. Clearly, \eqref{eq:arch} can be generalized by employing any number of layers, as well as any sequence of stepsizes $h_{1},\dots,h_{n}$ and supports $r_{1},\dots,r_{n-1}$. \newstuff{Similarly, as we shall demonstrate in our experiments, one may also modify \eqref{eq:arch} to handle functional inputs, e.g. by introducing a mesh-informed layer at the beginning of the architecture.} The choice of the hyperparameters remains problem specific, but a good rule of thumb is to decrease the supports as the mesh becomes finer, so that the network complexity is kept under control (cf. Theorem \ref{theorem:sparsity}).
\begin{figure}
    \centering
    \includegraphics[width = \textwidth]{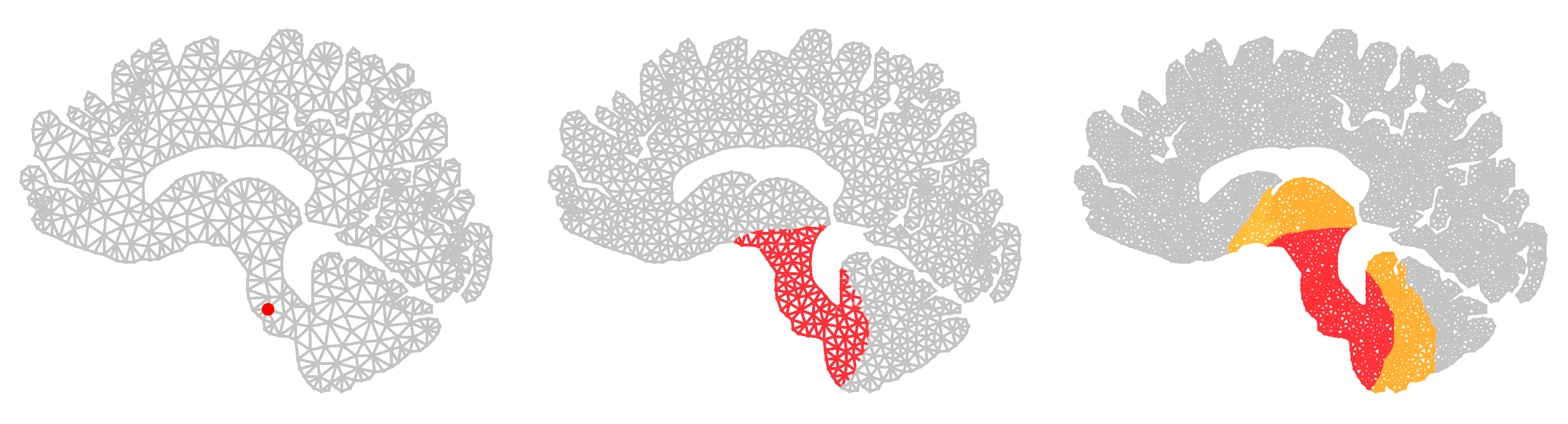}
    \caption{\newstuff{MINNs operate at different resolution levels to enforce local properties. Here, three meshes of a vertical brain slice, represent three different hidden states in the pipeline of a suitable MINN architecture, where neurons are identified with mesh vertices. Due to the sparsity constraint (Definition \ref{def:meshinformed}), the red neuron only fires information to those in the highlighted region (second mesh); in turn, these can only communicate with the neurons in the third mesh that are sufficiently close (orange region). In this way, despite relying on local operations, MINNs can eventually spread information all over the domain by exploiting the composition of multiple layers.}}
    \label{fig:meshes}
\end{figure}


\subsection{Implementation  details}\label{subsec:implementation}
\newstuff{For simplicity, let us first focus on the case in which distances are evaluated according to the Euclidean metric, $\geodistance(\x,\x')=|\x-\x'|.$
In this case, the practical implementation of mesh-informed layers is straightforward and can be done as follows, cf. Figure \ref{fig:master}.} Given $\Omega\subset\mathbb{R}^{d}$, $h,h'>0$, let $\mathbf{X}\in\mathbb{R}^{N_{h}\times d}$ and $\mathbf{X}'\in\mathbb{R}^{N_{h'}\times d}$ be the matrices containing the degrees of freedom associated to the chosen finite element spaces, that is $\mathbf{X}_{j,.}:=[X_{j,1},\dots,X_{j,d}]$ are the coordinates of the $j$th node in the input mesh, and similarly for $\mathbf{X}'$. In order to build a mesh-informed layer of support $r>0$, we first compute all the pairwise distances $D_{i,j}:=|\mathbf{X}_{j,.}-\mathbf{X}'_{i,.}|^{2}$ among the nodes in the two meshes. This can be done efficiently using tensor algebra, e.g.
\begin{equation*}D = \sum_{l=1}^{d}(e_{N_{h'}}\otimes\mathbf{X}_{.,l}-\mathbf{X}'_{.,l}\otimes e_{N_{h}})^{\circ 2}\end{equation*}
where $\mathbf{X}_{.,l}$ is the $l$th column of $\mathbf{X}$, $e_{k}:=[1,\dots,1]\in\mathbb{R}^{k}$, $\otimes$ is the tensor product and $\circ 2$ is the Hadamard power. We then extract the indices $\{(i_{k}, j_{k})\}_{k=1}^{\text{dof}}$ for which $D_{i_{k},j_{k}}\le r^{2}$ and initialize a weight vector $\mathbf{w}\in\mathbb{R}^{\text{dof}}$. This allows us to declare the weight matrix $\weight$ in sparse format by letting the nonzero entries be equal to $w_{k}$ at position $(i_{k},j_{k})$, so that $\|\weight\|_{0}=\text{dof}$. Preliminary to the training phase, we fill the \newstuff{entries of $\mathbf{w}$ with random values sampled from a suitable Gaussian distribution. The optimal choice for such distribution may be problem dependent. Empirically, we see that good results can be achieved by sampling the weights $w_{1},\dots,w_{\text{dof}}$ independently from a centered normal distribution with variance $1/\text{dof}$. Otherwise, if the architecture is particularly deep, another possibility is to consider an adaptation of the He initialization \cite{he}. In the case of $\alpha$-leakyReLU nonlinearities, the latter suggests sampling $w_{i}$ from a centered Gaussian distribution with variance
\begin{equation}\label{eq:he}\sigma_{i}^{2}=\frac{2}{(1+\alpha^{2})s_{i}},\end{equation}
where $s_{i}$ is the cardinality of the set $\{j\;|\;D_{i,j}\le r^{2}\}$, that is, the number of input neurons that communicate with the $i$th output neuron. Finally, in analogy to \cite{he}, if the layer has no activation then one may just set $\alpha=1$ in Equation \eqref{eq:he}.}\\\\
\begin{figure}
    \centering
    \includegraphics[width = \textwidth]{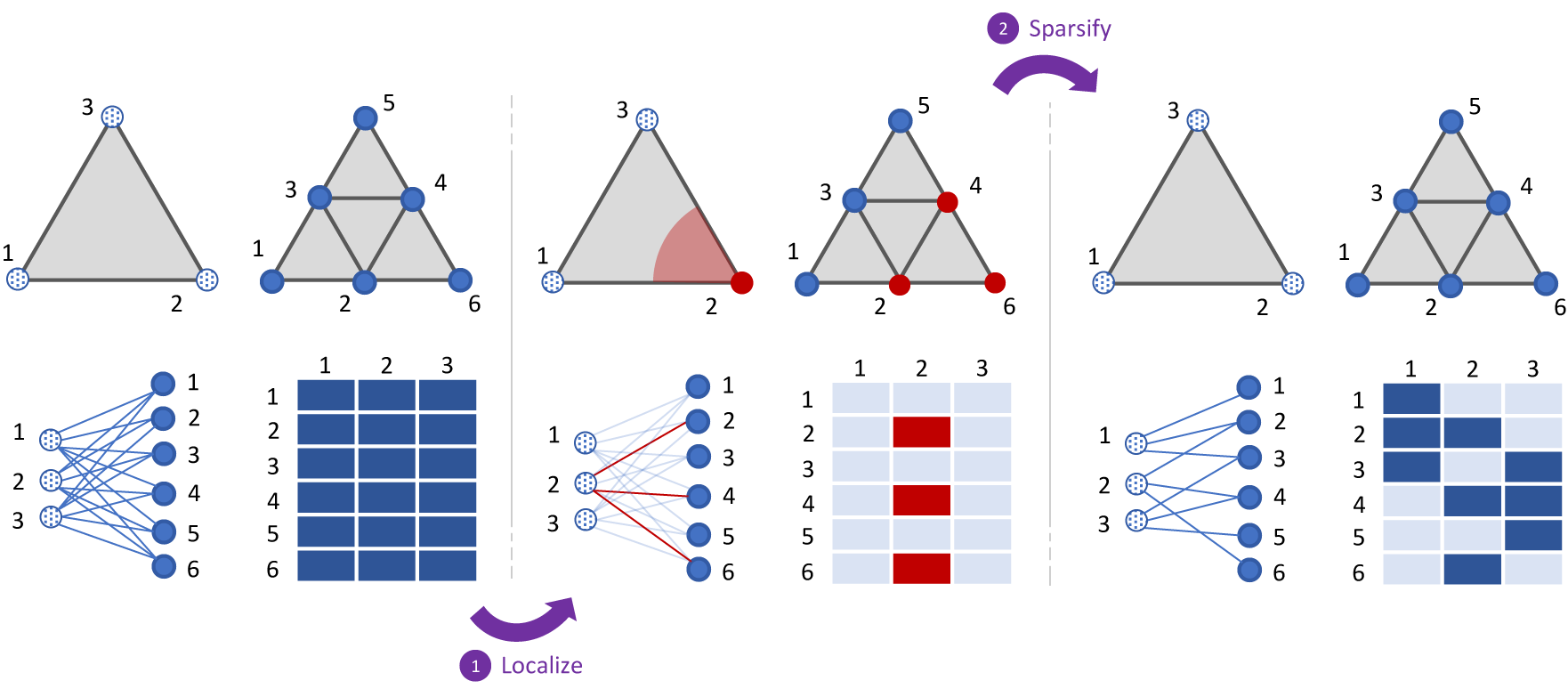}
    \caption{\newstuff{Implementation of a mesh-informed layer. Each panel reports a representation of the input/output meshes (top row), the network architecture (bottom left) and the corresponding weight matrix (bottom right). The degrees of freedom at the input/output meshes are associated to the input/output neurons of a reference dense layer (left panel). For each input node, the neighbouring nodes at output are highlighted (center panel). Then, all the remaining connections are pruned, and the final model is left with a sparse weight matrix (right panel).}}
    \label{fig:master}
\end{figure}

\noindent\newstuff{The above reasoning can be easily adapted to the general case, provided that one is able to compute efficiently all the pairwise distances $\geodistance(\mathbf{X}_{j,.},\mathbf{X}'_{i,.})$. Of course, the actual implementation will then depend on the specific choice of $\geodistance$. Since the case of geodesic distances can be of particular interest in certain applications,
we shall briefly discuss it below.
In this setting, the main difficulty arises from the fact that, in general, we are required to compute distances between points of different meshes. Additionally, if we consider Finite Element spaces of degree $q>1$, not all the Lagrangian nodes will be placed over the mesh vertices, meaning that we cannot exploit the graph structure of the mesh to calculate shortest paths.
\\
\indent To overcome these drawbacks, we propose the introduction of an auxiliary coarse mesh $\mesh_{0}:=\{K_{i}\}_{i=1}^{m}$, whose sole purpose is to capture the geometry of the domain. We use this mesh to build another graph, $\mathscr{G}$, which describes the location of the elements $K_{i}$. More precisely, let $\mathbf{c}_{i}$ be the centroid of the element $K_{i}$. We let $\mathscr{G}$ be the weighted graph having vertices $\{\mathbf{c}_{i}\}_{i=1}^{m}$, where we link $\mathbf{c}_{i}$ with $\mathbf{c}_{j}$ if and only if the elements $K_{i}$ and $K_{j}$ are adjacent. Then, to weight the edges, we use the Euclidean distance between the centroids. Once $\mathscr{G}$ is constructed, we use Dijkstra's algorithm to compute all the shortest paths along the graph. This leaves us with an estimated geodesic distance $g_{i,j}$ for each pair of centroids $(\mathbf{c}_{i},\mathbf{c}_{j})$, which we can precompute and store in a suitable matrix. Then, we approximate the  geodesic distance of \textit{any} two points $\x,\x'\in\Omega$ as
\begin{equation*}
    \geodistance(\x,\x')\approx g_{\mathpzc{i}(\x),\mathpzc{i}(\x')},
\end{equation*}
where $\mathpzc{i}:\Omega\to\{1,\dots,m\}$ maps each point to an element containing it, see Figure \ref{fig:geodesic}. In other words,
\begin{equation*}
    \mathpzc{i}(\x):=\min\{i\;|\;\x\in K_{i}\}.
\end{equation*}
If the auxiliary mesh does not fill $\Omega$ completely, the relaxed version below may be used as well
\begin{equation*}
    \mathpzc{i}(\x):=\min\left\{i\;|\;\dist(\x,K_{i})=\min_{j=1,\dots,m} \dist(\x,K_{j})\right\}.
\end{equation*}
In general, evaluating the index function $\mathpzc{i}$ can be done in $\mathcal{O}(m)$ time. In particular, going back to our original problem, we can approximate all the pairwise distances between the nodes of the input and the output spaces in $\mathcal{O}(mN_{h}+mN_{n}'+m^{2})$ time. In fact, we can run Dijkstra’s algorithm once and for all with a computational cost of $\mathcal{O}(m^{2})$. Then, we just need to evaluate the index function $\mathpzc{i}$ for all the Lagrangian nodes $\{\x\}_{i=1}^{N_{h}}$ and $\{\x'\}_{i=1}^{N_{h'}}$, which takes respectively $\mathcal{O}(mN_{h})$ and $\mathcal{O}(mN_{h'})$.} 
\begin{figure}
    \centering
    \includegraphics[width = 0.8\textwidth]{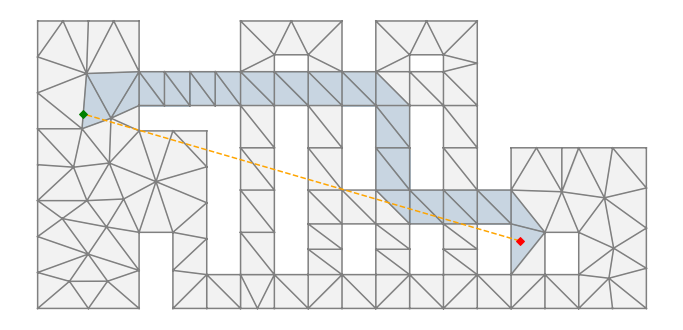}
    \caption{\newstuff{Approximation of the geodesic distance $\geodistance(\x,\y)$ as the shortest path across elements (blue triangles). In general, the latter will differ from the Euclidean distance between the two points (orange dashed line).}}
    \label{fig:geodesic}
\end{figure}

\newstuff{
\subsection{A supervised learning approach based on MINNs}
Once a mesh-informed architecture has been constructed, we learn the operator of interest through a standard supervised approach. Let $\operator_{h}:\Theta\to V_{h}$ be the high-fidelity discretization of a suitable operator,
whose pointwise evaluations may as well involve the numerical solution to a PDE. Here, we allow the input space, $\Theta$, to either consist of vectors, $\Theta\subset\mathbb{R}^{p}$, or (discretized) functions, e.g. $\Theta=V_{h}.$ Let $\Phi:\Theta\to V_{h}$ be a given MINN architecture: in general, depending on the input type, the latter will consist of both dense and mesh-informed layers. We assume that $\Phi$ has already been initialized according to some procedure, such as those reported in Section \ref{subsec:implementation}, and that it is ready for training.
\acapo
We make direct use of the forward operator to compute a suitable collection of training pairs (thereby also referred to as \textit{training samples} or \textit{snapshots}, following the classical terminology adopted in the Reduced Order Modeling literature),
$$\{f_{i},\;u_{h}^{i}\}_{i=1}^{N_{\text{train}}}\subset\Theta\times V_{h},$$
with $u_{h}^{i}:=\operator_{h}(f_{i}).$ Here, with little abuse of notation, we write $f\in\Theta$, that is, we assume to be in the case of functional inputs. If not so, the reader may simply replace the $f_{i}$'s with suitable $\mup_{i}$'s. The training pairs are typically chosen at random, i.e. by equipping the input space with a probability distribution $\mathbb{P}$ and by computing $N_{\text{train}}$ independent samples.
\acapo
We then train $\Phi$ by minimizing the mean squared $L^{2}$-error, that is, by optimizing the following loss function
\begin{equation}
    \label{eq:empirical_loss}
    \hat{\mathscr{L}}(\Phi)=\frac{1}{N_{\text{train}}}\sum_{i=1}^{N_{\text{train}}}\|u_{h}^{i}-\Phi(f_{i})\|_{L^{2}(\Omega)}^{2},
\end{equation}
which acts as the empirical counterpart of
\begin{equation}
    \label{eq:loss}
    \mathscr{L}(\Phi)=\mathbb{E}_{f\sim\mathbb{P}}\|\operator_{h}(f)-\Phi(f)\|_{L^{2}(\Omega)}^{2}.
\end{equation}
This allows one to actually tune the model parameters (i.e. all layers weights and biases), and obtain a suitable approximation $\Phi\approx\operator_{h}.$ Then, multiple metrics can be used to evaluate the quality of such an approximation. In this work, we shall often consider the average relative $L^{2}$-error, which we estimate thanks to a precomputed \textit{test set}
\begin{equation}
\label{eq:error}
\mathscr{E}(\Phi)=\frac{1}{N_{\text{test}}}\sum_{i=1}^{N_{\text{test}}}\frac{\|u_{h}^{i,\text{test}}-\Phi(f_{i}^{\text{test}})\|_{L^{2}(\Omega)}}{\|u_{h}^{i,\text{test}}\|_{L^{2}(\Omega)}}\approx\mathbb{E}_{f\sim\mathbb{P}}\left[\frac{||\operator_{h}f-\Phi(f)||_{L^{2}(\Omega)}}{||\operator_{h}f||_{L^{2}(\Omega)}}\right].
\end{equation}
The test set $\{f_{i}^{\text{test}},u_{h,\text{test}}^{i}\}_{i=1}^{N_{\text{test}}}$ is constructed independently of the training set, but still relying on the high-fidelity operator $\operator_{h}$ and the probability distribution $\mathbb{P}$.
\acapo 
If the approximation is considered satisfactory, then the expensive operator $\operator_{h}$ can be replaced with the cheaper MINN surrogate $\Phi$, whose outputs can be evaluated with little to none computational cost. 
\\\\
Before concluding this paragraph, it may be worth to emphasize a couple of things. First of all, we highlight the fact that both the loss and the error function, respectively $\hat{\mathscr{L}}$ and $\mathscr{E}$, require the computation of integral norms. However, this is not an issue: since both $u_{h}^{i}$ and $\Phi(f_{i})$ lie in $V_{h}$, we can compute these norms by relying on the mass matrix $\mathbf{M}\in\mathbb{R}^{N_{h}\times N_{h}}$, which can be precomputed and stored once and for all. We recall in fact that the latter is a highly sparse matrix defined in such a way that
$$\forall v\in V_{h},\;\;\mathbf{v}:=[v(\x_{1}),\dots,v(\x_{N_{h}})]^{T}\implies \|v\|_{L^{2}(\Omega)}=\mathbf{v}^{T}\mathbf{M}\mathbf{v},$$
where $\{\x_{i}\}_{i=1}^{N_{h}}$ are the nodes corresponding to the degrees of freedom in $V_{h}$, cf. \textit{function-to-nodes operator} in Definition \ref{def:fe}.
\acapo
Finally, we remark that the MINN architecture is trained in a purely supervised fashion. Even if the definition of the operator $\operator_{h}$ might involve a PDE or any other physical law, none of this knowledge is imposed over $\Phi$. Similarly, we do not impose boundary conditions, mass conservation, or other constraints, on the outputs of the MINN model: in principle, these should be learned implicitly. Nonetheless, we recognize that MINNs should benefit from the integration of such additional knowledge, as this is what other researchers have already observed for other approaches in the literature \cite{ChenPINN,Gao,paris}. As of now, we leave these considerations for future works.
}

\subsection{Relationship to other Deep Learning techniques}
It is worth to comment on the differences and similarities that MINNs share with other Deep Learning approaches. We discuss them below.

\subsubsection{Relationship to CNNs and GNNs}
Mesh-Informed architectures can operate at different \newstuff{levels of resolution}, in a way that is very similar to CNNs. However, their construction comes with multiple advantages. First of all, Definition \ref{def:meshinformed} adapts to any geometry, while convolutional layers typically operate on square or cubic input-output. Furthermore, convolutional layers use weight sharing, meaning that all parts of the domain are treated in the same way. This may not be an optimal choice in some applications, such as those involving PDEs, as we may want to differentiate our behavior over $\Omega$ (for instance near of far away from the boundaries). 

Conversely, MINNs share with GNNs the ability to handle general geometries. As a matter of fact, we mention that these architectures have been recently applied to mesh-based data, see e.g. \cite{meshnet,gruber,deepmind,meshcond}. With respect to GNNs, the main advantage of MINNs lies in their capacity to work at different \newstuff{resolutions}. This fact, which essentially comes from the presence of an underlying spatial domain, has at least two advantages: it \newstuff{provides a direct way for either reducing or increasing the dimensionality of a given input,} and it increases the interpretability of hidden layers (now the number of neurons is not arbitrary but comes from the chosen discretization). \newstuff{In a way, mesh-informed architectures are similar to hierarchical GNNs, however, their construction is fundamentally different. While GNNs
use \textit{aggregation} strategies to collapse neighbouring information, MINNs differentiate their behavior depending on the nodes that are being involved. This difference, makes the two approaches better suited for different applications. For instance, GNNs can be of great interest when dealing with geometric variability, as their construction allows for a single architecture to operate over completely different domains, see e.g. \cite{marta}. Conversely, if the shape of the domain is fixed, then MINNs may grant a higher expressivity, as they are ultimately obtained by pruning dense feedforward models.}
In this sense, MINNs \newstuff{only} exploit meshing strategies as auxiliary tools, and they appear to be a natural choice for learning discretized functional outputs.

\subsubsection{Relationship to DeepONets and Neural Operators}
\label{subsec:relationship}
Recently, some new DNN models have been proposed for operator learning. One of these are DeepONets \cite{karniadakis}, a mesh-free approach that is based on an explicit decoupling between the input and the space variable. More precisely, DeepONets consider a representation of the following form
$$(\mathcal{G}_{h}f)(\x)\approx\Psi(f)\cdot \phi(\x),$$
where $\cdot$ is the dot product, $\Psi: V_{h}\to \mathbb{R}^{m}$ is the \newstuff{branch-net}, and $\phi: \Omega\to\mathbb{R}^{m}$ is the \newstuff{trunk-net}. DeepONets have been shown capable of learning nonlinear operators and are now being extended to include apriori physical knowledge, see e.g. \cite{paris}.
We consider MINNs and DeepONets as two fundamentally different approaches that answer different questions. DeepONets were originally designed to process input data coming from sensors and, being mesh-free, they are particularly suited for those applications where the output is only partially known or observed. In contrast, MINNs are rooted on the presence of a high-fidelity model $\operator_{h}$ and are thus better suited for tasks such as reduced order modeling. Another difference lies in the fact that DeepONets include explicitly the dependence on the space variable $\x$\newstuff{: because of this, suitable strategies are required in order for them to deal with complex geometries or incorporate additional information, such as boundary data, see e.g. \cite{goswami}.} Conversely, MINNs can easily handle this kind of issues thanks to their global perspective.

In this sense, MINNs are much closer to the so-called Neural Operators, a novel class of DNN models first proposed by Kovachki et al. in \cite{kovachki}. Neural Operators are an extension of classical DNNs that was developed to operate between infinite dimensional spaces.
These models are ultimately based on Hilbert-Schmidt operators, meaning that their linear part, that is ignoring activations and biases, is of the form
\begin{equation}
    \label{eq:nops}
    W: f\to \int_{\Omega}k(\cdot,\x)f(\x)d\x
\end{equation}
where $k:\Omega\times\Omega\to\mathbb{R}$ is some \textit{kernel} function that has to be identified during the training phase.
Clearly, the actual implementation of Neural Operators is carried out in a discrete setting and integrals are replaced with suitable quadrature formulas. \newstuff{We mention that, among these architectures, Fourier Neural Operators (FNO) are possibly the most popular ones: we shall discuss more about the comparison between MINNs and FNOs in Section \ref{sec:experiments4}.}

The construction of Neural Operators is \newstuff{very general, to the point that other approaches, such as ours, can be seen as a special case.} Indeed, a rough Monte Carlo-type estimate of \eqref{eq:nops} would yield
$$(Lf)(\x'_{i})=\int_{\Omega}k(\x'_{i},\x)f(\x)d\x\approx\frac{|\Omega|}{\fomdim}\sum_{j}k(\x'_{i},\x_{j})f(\x_{j}).$$
If we let $\{\x_{j}\}_{j}$ and $\{\x'_{i}\}_{i}$ be the nodes in the two meshes, then the constraint in Definition 5 becomes equivalent to the requirement that $k$ is supported somewhere near the diagonal, that is $\text{supp}(k)\subseteq\{(\x,\x+\mathbf{\varepsilon})\;\text{with}\;|\mathbf{\varepsilon}|\le r\}$. 
\\\\
We believe that these parallels are extremely valuable, as they indicate the existence of a growing scientific community with common goals and interests. \newstuff{Furthermore, they all contribute to the enrichment of the operator learning literature, providing researchers with multiple alternatives from which to choose: in fact, as we shall see in Section \ref{sec:experiments4}, depending on the problem at hand one methodology might be better suited than the others.}

\section{Numerical experiments: effectiveness of the pruning strategy}
\label{sec:experiments1}
We provide empirical evidence that the sparsity introduced by MINNs can be of great help in learning maps that involve functional data, such as nonlinear operators, showing a reduced computational cost and better generalization capabilities. We first detail the setting of each experiment alone, specifying the corresponding operator to be learned and the adopted MINN architecture. Then, at the end of the current Section, we discuss the numerical results.

Throughout all our experiments, we adopt a standardized approach for designing and training the networks. In general, we always employ the 0.1-leakyReLU activation for all the hidden layers, while we do not use any activation at the output. Every time a mesh-informed architecture is introduced, we also consider its dense counterpart, obtained without imposing the sparsity constraints. Both networks are then trained following the same criteria, so that a fair comparison can be made.
As loss function, we always consider the mean square error computed with respect to the $L^{2}$ norm, cf. Equation \eqref{eq:loss}. 

The optimization of the loss function is performed via the L-BFGS optimizer, with learning rate always equal to 1 and no batching. What may change from case to case are the network architecture, the number of epochs, and the size of the training set. After training, we compare mesh-informed and dense architectures by evaluating their performance on 500 unseen samples (test set), which we use to compute an unbiased estimate of the average relative error, cf. Equation \eqref{eq:error}.

\noindent\newline All the code was written in Python 3, mainly relying on \newstuff{the \texttt{FEniCS} 
and \texttt{mshr} libraries 
for the construction of Finite Element spaces and the numerical solution of PDEs. The implementation of neural network models, instead, was handled in \texttt{Pytorch}, 
and their training was carried out on a GPU NVidia Tesla V100 (32GB of RAM).}
\begin{figure}
\newcommand{\spaziatura}{\hspace{8.5em}}
\flushleft
\hspace{1em}a)\spaziatura b)\spaziatura c)\spaziatura d)
\includegraphics[width=\textwidth]{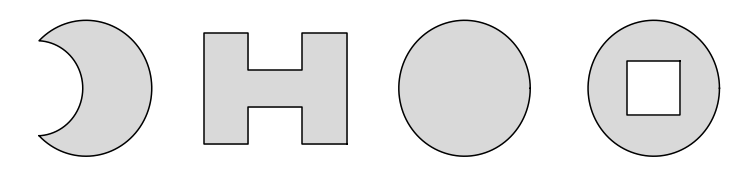}
\caption{\label{fig:domains}Domains considered for the numerical experiments in Section \ref{sec:experiments1}. Panel a): up to boundaries, $\Omega$ is the difference of two circles, $B(\textbf{0}, 1)$ and $B(\x_{0},0.7)$, where $\x_{0}=(-0.75, 0)$. Panel b): A polygonal domain obtained by removing the rectangles $[-0.75, 0.75]\times[0.5, 1.5]$ and $[-0.75,0.75]\times[-1.5,-0.5]$ from $(-2,2)\times(-1.5, 1.5)$. Panel c): the unit circle $B(\textbf{0},1)$. Panel d): $\Omega$ is obtained by removing a square, namely $[-0.4,0.4]^{2}$, from the unit circle $B(\textbf{0},1)$.}\end{figure}

\subsection{Description of the benchmark operators}

\subsubsection*{Learning a parametrized family of functions}
Let $\Omega$ be the domain defined as in Figure 2a. For our first experiment, we consider a variation of a classical problem concerning the calculation of the \textit{signed distance function} of $\Omega$. This kind of functions are often encountered in areas such as computer vision \cite{computervision} and real-time rendering \cite{rendering}. In particular, we consider the following operator, 
$$\operator:\Theta\subset\mathbb{R}^{3}\to L^{2}(\Omega)$$
$$\operator:\mup\to u_{\mup}(\x):=\min_{\substack{\y\in\partial\Omega,\\y_{2}>\mu_{1}}}\;|\y-\mathbf{A}_{\mu_{3}}\x|\textrm{e}^{x_{1}\mu_{2}}$$
where $\mup=(\mu_1,\mu_2,\mu_3)$ is a finite dimensional vector, and $\mathbf{A}_{\mu_3} = \text{diag}(1,\mu_{3})$. In practice, the value of $u_{\mup}(\x)$ corresponds to the (weighted) distance between the dilated point $\mathbf{A}_{\mu_{3}}\x$ and the truncated boundary $\partial\Omega\cap\{\y:\;y_{2}>\mu_{1}\}$.

As input space we consider $\Theta:=[0,1]\times[-1,1]\times[1,2]$, endowed with the uniform probability distribution. Since the input is finite-dimensional, we can think of $\operator$ as to the parametrization of a 3-dimensional hypersurface in $L^{2}(\Omega)$. We discretize $\Omega$ using P1 triangular finite elements with mesh stepsize $h=0.02$, resulting in the high-fidelity space $V_{h}:=\cong\mathbb{R}^{13577}$. To learn the discretized operator $\operator_{h}$, we employ the following MINN architecture
$$\mathbb{R}^{3}\to\mathbb{R}^{100}\to V_{9h}\xrightarrow{\;r=0.4\;}V_{3h}\xrightarrow{\;0.2\;}V_{h},$$
\newstuff{where the supports are defined according to the Euclidean distance.}
The corresponding dense counterpart, which servers as benchmark, is obtained by removing the sparsity constraints (equivalently, by letting the supports go to infinity). We train the networks on 50 samples and for a total of 50 epochs.

\subsubsection*{Learning a local nonlinear operator}
As a second experiment, we learn a nonlinear operator that is local with respect to the input. Let $\Omega$ be as in Figure 2b. We consider the \textit{infinitesimal area} operator $\operator: H^{1}(\Omega)\to L^{2}(\Omega)$,
$$\operator: u\to\sqrt{1+|\nabla u|^{2}}.$$
Note in fact that, if we associate to each $u\in H^{1}(\Omega)$ the cartesian surface $$S_{u}:=\{(x_{1}, x_{2}, u(x_{1}, x_{2})\}\subset\mathbb{R}^{3},$$ then $\operator u$ yields a measure of the area that is locally spanned by that surface, in the sense that
$$\int_{S_{u}}\phi(\s)d\s=\int_{\Omega}\phi(u(\x))\left(\operator u\right)(\x)d\x$$
for any continuous map $\phi:S_{u}\to\mathbb{R}.$
Over the input space $\Theta:=H^{1}(\Omega)$ we consider the probability distribution $\mathbb{P}$ induced by a Gaussian process with mean zero and covariance kernel
$$\text{Cov}(\x, \y) = \frac{1}{|\Omega|}\exp\left(-\frac{1}{2}|\x-\y|^{2}\right).$$
We discretize $\Omega$ using a triangular mesh of stepsize $h=0.045$ and P1 finite elements, which results in a total of $\fomdim = 11266$ vertices. To sample from the Gaussian process we truncate its Karhunen-Loeve expansion at $k=100$. Conversely, the output is computed numerically by exploiting the high-fidelity mesh as a reference.

To learn $\operator_{h}$ we use the MINN architecture below,
$$V_{h}\xrightarrow{\;r=0.15\;}V_{3h}\xrightarrow{\;r=0.3\;}V_{3h}\xrightarrow{\;r=0.15\;}V_{h},$$
\newstuff{where the supports are intended with respect to the Euclidean metric. We train our model over 500 snapshots and for a total of 50 epochs.} 

\subsubsection*{Learning a nonlocal nonlinear operator}
Since MINNs are based on local operations, it is of interest to assess whether they can also learn nonlocal operators. To this end, we consider the problem of learning the Hardy-Littlewood Maximal Operator $\operator: L^{2}(\Omega)\to L^{2}(\Omega)$,
$$\left(\operator f\right)(\x):=\sup_{r>0}\;\strokedint_{|\y-\x|<r}|f(\y)|d\y$$
which is known to be a continuous nonlinear operator from $L^{2}(\Omega)$ onto itself \cite{hardy}. Here, we let $\Omega:=B(\mathbf{0},1)\subset\mathbb{R}^{2}$ be the unit circle. Over the input space $\Theta:=L^{2}(\Omega)$ we consider the probability distribution $\mathbb{P}$ induced by a Gaussian process with mean zero and covariance kernel
$$\text{Cov}(\x, \y) = \exp\left(-|\x-\y|^{2}\right).$$
As a high-fidelity reference, we consider a discretization of $\Omega$ via P1 triangular finite elements of stepsize $h=0.033$, resulting in a state space $V_{h}$ with $N_{h}=7253$ degrees of freedom. As for the previous experiment, we sample from $\mathbb{P}$ by considering a truncated Karhunen-Loeve expansion of the Gaussian process (100 modes). Conversely, the true output $u\to\operator_{h}(u)$ is computed approximately by replacing the suprema over $r$ with a maxima across 50 equispaced radii in $[h,2]$. To learn $\operator_{h}$ we use a MINN of depth 2 with a dense layer in the middle,
$$V_{h}\xrightarrow{\;r=0.25\;} V_{2h}\to V_{2h}\xrightarrow{\;r=0.25\;}V_{h}.$$
The idea is that nonlocality can still be enforced through the use of fully connected blocks, but this are only inserted at the lower fidelity levels to reduce the computational cost. We train the architectures over 500 samples and for a total of 50 epochs. \newstuff{Also in this case, we build the mesh-informed layers upon the Euclidean distance.}

\subsubsection*{Learning the solution operator of a nonlinear PDE} For our final experiment, we consider the case of a parameter dependent PDE, which is a framework of particular interest in the literature of Reduced Order Modeling. In fact, learning the solution operator of a PDE model by means of neural networks allows one to replace the original numerical solver with a much cheaper and efficient surrogate, which enables expensive multi-query tasks such as PDE constrained optimal control, Uncertainty Quantification or Bayesian Inversion.

Here, we consider a steady version of the  porous media equation, defined as follows
\begin{equation}
\label{eq:nonlinear}
-\nabla\cdot\left(|u|^{2}\nabla u\right) + u = f.\end{equation}
The PDE is understood in the domain $\Omega$ defined in Figure 2d, and it is complemented with homogeneous Neumann boundary conditions. We define $\operator$ to be the data-to-solution operator that maps $f\to u$. 
This time, we endow the input space with the push-forward distribution $\#\mathbb{P}$ induced by the square map $f\to f^{2}$, where $\mathbb{P}$ is the probability distribution associated to a Gaussian random field with mean zero and covariance kernel
$$\text{Cov}(\x,\y)=\frac{1}{1+|\x-\y|^{2}}.$$
To sample from the latter distribution we exploit a truncated Karhunen-Loeve expansion of the random field. We set the truncation index to $k=20$ as that is sufficient to fully capture the volatility of the field. For the high-fidelity discretization, we consider a mesh of stepsize $h=0.03$ and P1 finite elements, resulting in $N_{h}=5987$.
Finally, we employ the MINN architecture below,
\begin{equation}
\label{eq:archnonlin}
V_{h}\xrightarrow{\;r=0.4\;} V_{4h}\to V_{2h}\xrightarrow{\;r=0.2\;}V_{h},\end{equation}
\newstuff{where the supports are defined according to the Euclidean distance. We train our model} over 500 snapshots and for a total of 100 epochs. Note that, as in our third experiment, we employ a dense block at the center of the architecture. This is because the solution operator to a boundary value problem is typically nonlocal (consider, for instance, the Green formula for the Poisson equation).

\subsection{Numerical results}
\renewcommand{\arraystretch}{2.5}
\newcommand{\tb}[1]{\Centerstack[l]{#1}}
\newcommand{\dcell}[2]{\tb{#1\\#2}}
\begin{table}
 \begin{tabular}{| l | l | l | l | l |} 
 \hline
  Operator & $\fomdim$ & Architecture & Test error & Gen. error\\
 \hline\hline
 \tb{Low-dimensional\\manifold} & 13'577 & \dcell{Mesh-Informed}{Dense} & \dcell{4.14\%}{4.78\%} & \dcell{3.08\%}{4.07\%}\\
 \hline
 \tb{Local area operator}   & 11'266 &\dcell{Mesh-Informed}{Dense} & \dcell{1.49\%}{3.89\%} & \dcell{1.36\%}{2.54\%}\\
 \hline
 \tb{H-L maximal operator} & 7'253 &\dcell{Mesh-Informed}{Dense} & \dcell{3.65\%}{6.70\%} & \dcell{1.49\%}{3.09\%}\\
 \hline
 \tb{Nonlinear PDE solver} & 5'987 &\dcell{Mesh-Informed}{Dense} &  \dcell{4.29\%}{15.67\%} &  \dcell{3.03\%}{6.67\%}\\
 \hline
\end{tabular}
\vspace{1em}
\caption{\label{tab:minns} 
Comparison of Mesh-Informed Neural Networks and Fully Connected DNNs for the test cases in Section \ref{sec:experiments1}. $\fomdim$ = number of vertices in the (finest) mesh. Gen. error = Generalization error, defined as the gap between training and test errors. All reported errors are intended with respect to the $L^{2}$-norm, see Equation \eqref{eq:error}.}
\end{table}
\begin{figure}
\includegraphics[width=0.85\textwidth]{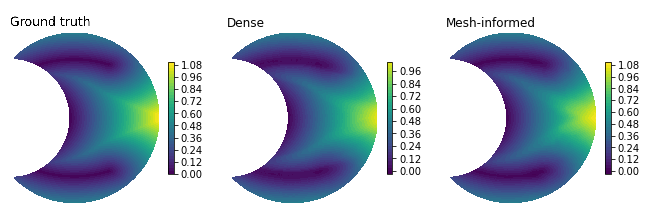}
\caption{\label{fig:testminns1}Comparison of DNNs and MINNs when learning a low-dimensional manifold $\mup\to u_{\mup}\in L^{2}(\Omega)$, cf. Section \ref{sec:experiments1}.1. The reported results correspond to the approximations obtained on an unseen input value $\mup^{*}=[0.42, 0.04, 1.45].$}\end{figure}
\begin{figure}
\includegraphics[width=0.95\textwidth]{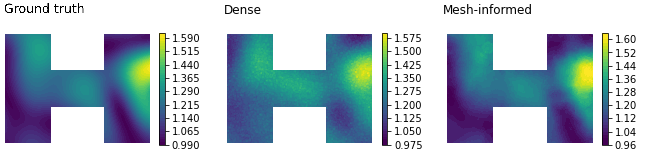}
\caption{\label{fig:testminns2}Comparison of DNNs and MINNs when learning the local operator $u\to\sqrt{1+|\nabla u|^{2}}$, cf. Section \ref{sec:experiments1}.1. The reported results correspond to the approximations obtained for an input instance outside of the training set.}\end{figure}
\begin{figure}
\includegraphics[width=0.95\textwidth]{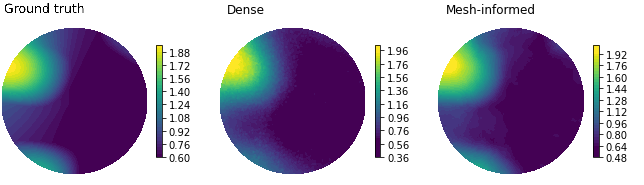}
\caption{\label{fig:testminns3}Comparison of DNNs and MINNs when learning the Hardy-Littlewood Maximal Operator, cf. Section \ref{sec:experiments1}.1. The pictures correspond to the results obtained for an unseen input instance.}\end{figure}
\begin{figure}
\includegraphics[width=0.95\textwidth]{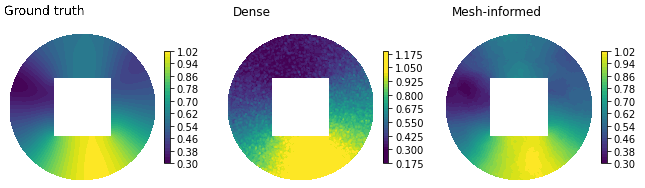}
\caption{\label{fig:testminns4}Comparison of DNNs and MINNs when learning the solution operator $f\to u$ of a nonlinear PDE, cf. Section \ref{sec:experiments1}.1. The reported results correspond to the approximations obtained for an input instance $f$ outside of the training set.}\end{figure}
Table \ref{tab:minns} reports the numerical results obtained across the four experiments. In general, MINNs perform better with respect to their dense counterpart, with relative errors that are always below 5\%. As the operator to be learned becomes more and more involved, fully connected DNNs begin to struggle, eventually reaching an error of 15\% in the PDE example. In contrast, MINNs are able to keep up and maintain their performance. This is also due to the fact that, having less parameters, MINNs are unlikely to overfit, instead they can generalize well even in poor data regimes (cf. last column of Table \ref{tab:minns} Consider for instance the first experiment, which counted as little as 50 training samples. There, the dense model returns an error of 4.78\%, of which 4.07\% is due to the generalization gap. This means that the DNN model actually surpassed the MINN performance over the training set, as their training errors are respectively of $0.71\%$  and $1.06\%$. However, the smaller generalization gap allows the sparse architecture to perform better over unseen inputs.

Figures \ref{fig:testminns1} to \ref{fig:testminns4} reports some examples of approximation on unseen input values. There, we note that dense models tend to have noisy outputs (Figures \ref{fig:testminns2} and \ref{fig:testminns4}) and often miscalculate the range of values spanned by the output (Figures \ref{fig:testminns1} and \ref{fig:testminns3}). Conversely, MINNs always manage to capture the main features present in the actual ground truth.
\newstuff{This goes to show that MINNs are built upon an effective pruning strategy, thanks to which they are able to overcome the limitations entailed by dense architectures. This phenomenon can be further appreciated in the plot reported in Figure \ref{fig:supports}.
The latter refers to the nonlinear PDE example, Equation \eqref{eq:nonlinear}, where our MINN architecture was of the form
\begin{equation}
\label{eq:deltaminn}
V_{h}\xrightarrow{\;r=\delta\;} V_{4h}\to V_{2h}\xrightarrow{\;r=\frac{1}{2}\delta\;}V_{h},\end{equation}
with $\delta=0.4$. Figure \ref{fig:supports} shows how the choice of the support size, $\delta$, affects the performance of the MINN model. Since, here, $\Omega$ has a diameter of 2, for $\delta=4$ the architecture is formally equivalent to a dense DNN. As the support decreases, the first and the last layer become sparser, and, in turn, the architecture starts to generalize more, ultimately reducing the test error by nearly 11\%. However, if the support is too small, e.g. $\delta\le0.125$, this may have a negative effect on the expressivity of the architecture, which now has too few degrees of freedom to properly approximate the operator of interest. In general, the best support size is to be found in the middle, $0\le\delta^{*}\le\diam(\Omega)$: 
 of note, in this case, our original choice happens to be nearly optimal.
}
\begin{figure}
\begin{center}
\includegraphics[width=0.7\textwidth]{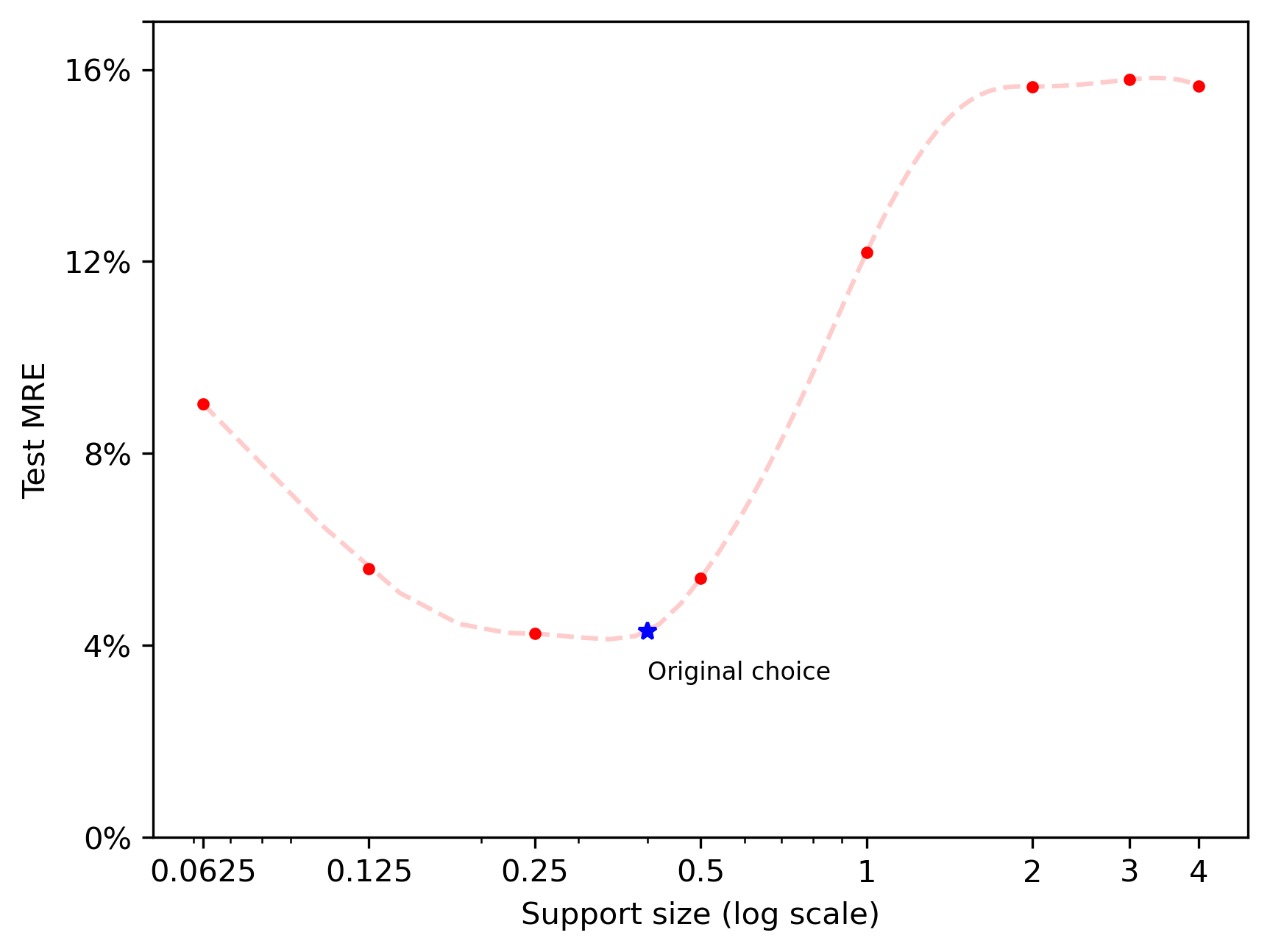}
\end{center}
\caption{\newstuff{\label{fig:supports}Relationship between model accuracy and support size in mesh-informed layers. Case study: nonlinear PDE, see Equation \eqref{eq:nonlinear}; MINN architecture: see Equation \eqref{eq:deltaminn}. Red dots correspond to different choices for the support of the mesh-informed layers; cubic splines are used to draw the general trend (red dashed line). In blue, our original choice for the architecture, see Equation \eqref{eq:archnonlin}. Errors are computed according to the $L^{2}$-norm. The $x$-axis is reported in logarithmic scale.}}\end{figure}
\renewcommand{\arraystretch}{2.5}
\newcommand{\ghost}[1]{\textcolor{white}{#1}}
\begin{table}[tb]
 \begin{tabular}{| l | r | r | rr |} 
 \hline\vspace{-1.5em}
Architecture & dof & \tb{Training speed} & Memory usage &\\

 &&&(static)&(training)\\
 \hline\hline
 \dcell{Mesh-Informed}{Dense} & \dcell{\ghost{0}0.3M}{21.7M} & \newstuff{\dcell{0.64 s/ep}{2.47 s/ep}} & \dcell{\ghost{0}1.3 Mb}{86.9 Mb} & \dcell{0.01 Gb}{3.49 Gb} \\
 \hline
 \dcell{Mesh-Informed}{Dense} & \dcell{\ghost{0}0.3M}{31.0M} & \newstuff{\dcell{1.21 s/ep}{3.89 s/ep}} & \dcell{\ghost{00}1.0 Mb}{124.1 Mb} & \dcell{0.04 Gb}{4.96 Gb} \\
 \hline
 \dcell{Mesh-Informed}{Dense} & \dcell{\ghost{0}1.0M}{12.8M} & \newstuff{\dcell{0.68 s/ep}{1.66 s/ep}} & \dcell{\ghost{0}3.9 Mb}{51.2 Mb} & \dcell{0.16 Gb}{2.05 Gb} \\
 \hline
 \dcell{Mesh-Informed}{Dense} &  \dcell{\ghost{0}1.4M}{12.5M} & \newstuff{\dcell{0.70 s/ep}{1.31 s/ep}} & \dcell{\ghost{0}5.4 Mb}{50.0 Mb} & \dcell{0.22 Gb}{2.00 Gb}\\
 \hline
\end{tabular}
\vspace{1em}
\caption{\label{tab:minnscosts} 
Comparison of Mesh-Informed Neural Networks and Fully Connected DNNs in terms of their computational cost. dof = degrees of freedom, i.e. number of parameters to be optimized during the training phase. Memory usage (static) = bytes required to store the architecture. Memory usage (optimization) =  bytes required to run a single epoch of the training phase. s/ep = seconds per epoch, M = millions, Mb = Megabytes, Gb = Gigabytes.}
\end{table}
\\\\
\newstuff{Aside from the improvement in performance}, Mesh-Informed Neural Networks also allow for a significant reduction in the computational cost (cf. Table \ref{tab:minnscosts}). In general, MINNs are ten to a hundred times lighter with respect to fully connected DNNs. While this is not particularly relevant once the architecture is trained (the most heavy DNN weights as little as 124 Megabytes), it makes a huge difference during the training phase. In fact, additional resources are required to optimize a DNN model, as one needs to keep track of all the operations and gradients in order to perform the so-called \textit{backpropagation} step. This poses a significant limitation to the use of dense architectures, as the entailed computational cost can easily exceed the capacity of modern GPUs. For instance, in our experiments, fully connected DNNs required more than 2 GB of memory during training, while, depending on the operator to be learned, 10 to 250 MB were sufficient for MINNs. Clearly, one could also alleviate the computational burden by exploiting cheaper optimization routines, such as first order optimizers and batching strategies, however this typically prevents the network from actually reaching the global minimum of the loss function. In fact, we recall that the optimization of a DNN architecture is, in general, a non-convex and ill-posed problem. Of note, despite being 10 to 100 times lighter, MINNs are only 2 to 4 times faster during training. We believe that these results can be improved, possibily by optimizing the code used to implement MINNs.

\newstuff{\section{Numerical experiments: handling general nonconvex domains}
\label{sec:experiments2}
In the previous Section, when building our mesh-informed architectures, we always referred to the Euclidean distance. This is because, for our analysis, we considered spatial domains that were still quite simple in terms of shape and topology. In this sense, while most of those domains were nonconvex, the Euclidean metric was still satisfactory for quantifying distances.
\acapo
In this Section, we would like to investigate this further, by testing MINNs over geometries that are far more complicated. We shall provide two examples: one concerning the flow of a Newtonian fluid through a system of channels, and one describing a reaction-diffusion equation solved on a 2D section of the human brain.
In both cases, we shall not provide comparisons with other Deep Learning techniques: the only purpose for this Section is to assess the ability of MINNs in handling complicated geometries.}

\newstuff{\subsection{Stokes flow in a system of channels}
\label{subsec:stokes}
As a first example, we consider the solution operator to the following parametrized (stationary) Stokes equation,
\begin{equation}
    \label{eq:stokes}
    \begin{cases}
    -\Delta\mathbf{u}+\nabla p = 0 & \text{in}\;\Omega\\
    -\nabla\cdot\mathbf{u}=0 & \text{in}\;\Omega\\
    p=0 & \text{on}\;\Gamma_{\text{out}}\\
    \mathbf{u}=0 & \text{on}\;\Omega\setminus(\Gamma_{\text{in}}\cup\Gamma_{\text{in}}\cup_{i=1}^{4}\Gamma_{i})\\
    \mathbf{u}=\mathbf{f}_{y} & \text{on}\;\Gamma_{\text{in}}\\
    \mathbf{u}=\mathbf{g}^{(i)}_{a_{i}} & \text{on}\;\Gamma_{i},\;\;i=1,\dots,4,
    \end{cases}
\end{equation}
where $p$ and $\mathbf{u}$ are respectively the fluid pressure and velocity, while $\Omega\subset\mathbb{R}^{2}$ is the domain depicted in Figure \ref{fig:pipes}.
The PDE depends on five scalar parameters. The first one, $y\in[0.1, 0.8]$ is used to parametrize the location of the inflow condition, which  is given as
\begin{equation*}
    \mathbf{f}_{y}(\mathbf{x}) = 100\cdot\left[\mathbf{1}_{y-0.1,y+0.1}(x_{2})((x_{2}-y)^{2}+0.01),\;0\right]
\end{equation*}
In other words, $\mathbf{f}_{y}$ describes a jet flow centered at $y$ and directed towards the right. Viceversa, the remaining four parameters, $a_{1},\dots,a_{4}\in[0,1]$, are used to model the presence of possible leaks at the corresponding regions $\Gamma_{1},\dots,\Gamma_{4}$. In particular, for each $i=1,\dots,4$,
\begin{equation*}
    \mathbf{g}^{(i)_{a_{i}}}(\x)=a_{i}\eta_{i}(x_{1})(1-\eta_{i}(x_{1}))\cdot\mathbf{n},
\end{equation*}
where $\mathbf{n}$ is the external normal, while the $\eta_{i}$'s are suitable affine transformations from $\mathbb{R}\to\mathbb{R}$ such that $\eta_{i}(\{x_{1}\;|\;\x\in\Gamma_{i}\})=[0,1].$ 
Our output of interest, instead, is the velocity vector field $\mathbf{u}:\Omega\to\mathbb{R}^{2}$, meaning that the operator under study is
\begin{equation*}
    \operator:(y,a_{1},\dots,a_{4})\to\mathbf{u}. 
\end{equation*}
\begin{figure}
    \centering
    \includegraphics[width=0.8\textwidth]{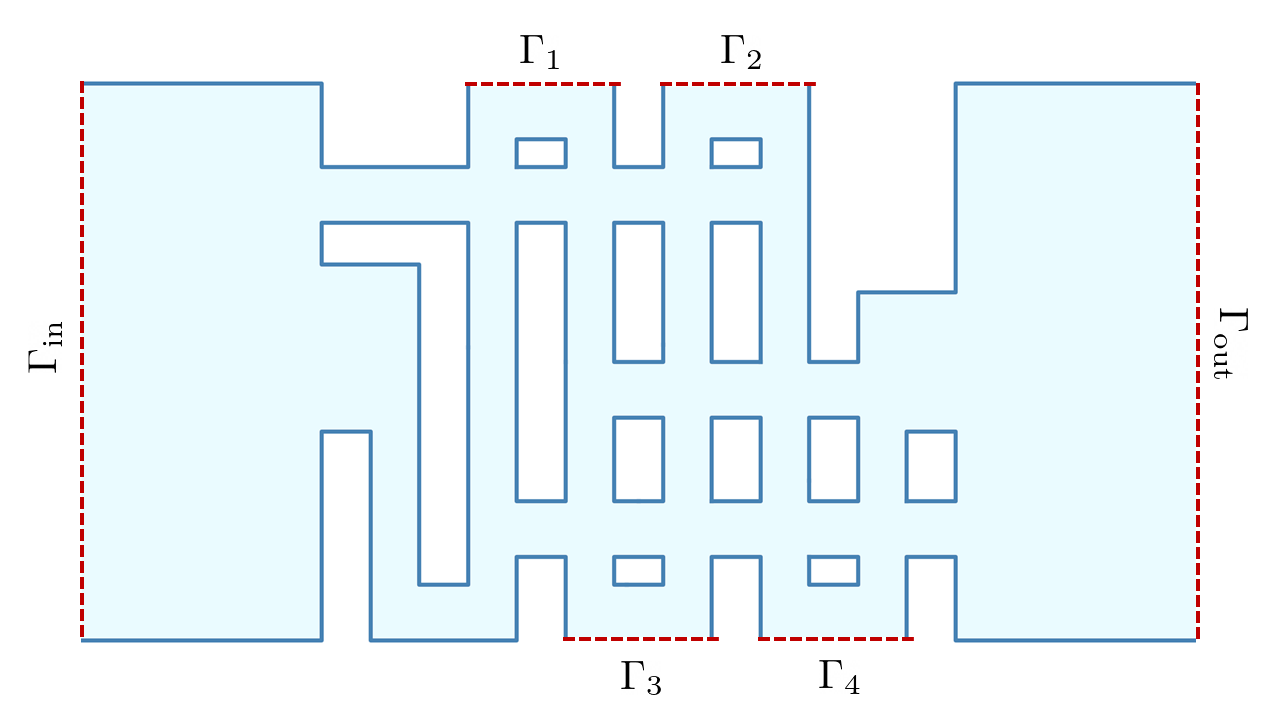}
    \caption{\newstuff{Spatial domain for the Stokes flow example, Section \ref{subsec:stokes}.}}
    \label{fig:pipes}
\end{figure}
\noindent As high-fidelity reference, we consider a discretized setting where the domain $\Omega$ is endowed with a triangular mesh of stepsize $h=0.08,$ the pressure $p$ is sought in the Finite Element space of piecewise-linear polynomials, $Q_{h}=X_{h}^{1}$, while the velocity $\mathbf{u}$ is found in the space $V_{h}$ of P1-Bubble vector elements \cite{quarteroni}, i.e. $X_{h}^{1}\times X_{h}^{1} \subset V_{h}\subset X_{h}^{3}\times X_{h}^{3}$. This is done in order to ensure the numerical stability of classical finite element schemes associated to \eqref{eq:stokes}. Then, our purpose is to approximate the operator $$\operator_{h}:[0.1,1.9]\times[0,1]^{4}\to V_{h}\cong\mathbb{R}^{17'050}$$ with a suitable MINN surrogate, which we build as
\begin{equation}
    \label{eq:stokesminn}
    \mathbb{R}^{5}\to\mathbb{R}^{100}\to V_{9h}\xrightarrow{\;r=1\;} V_{h}\xrightarrow{\;r=0.08\;}V_{h}.
\end{equation}
Note that, even for the hidden state at the third/fourth layer, we consider a space comprised of Bubble elements (however, this is not mandatory). Here, in order to handle the fact that $V_{h}$ consists of 2-dimensional vector fields, each Lagrangian node is counted twice during the construction of the mesh-informed layers. As a consequence, the mesh-informed layers will allow information to travel locally in space but also interact in between dimensions.
\acapo
As usual, we employ the 0.1-leakyReLU nonlinearity at the internal layers, while we do not use any activation at output. We train our architecture over 450 randomly sampled snapshots, and for a total of 50 epochs. Following the same criteria as in Section \ref{sec:experiments1}, we train the network by optimizing the average squared error with respect to the $L^{2}$-norm, cf. Equation \ref{eq:loss}, without using batching strategies and by relying on the L-BFGS optimizer.
This time, however, in order to account for the complicated shape of the spatial domain $\Omega$, we impose the sparsity constraints in the mesh-informed layers according to the geodesic distance. We compute the latter as in Section \ref{subsec:implementation}, where we exploit a coarse mesh of stepsize $h^{*}=0.2$ to approximate distances.
\\\\
After training, our MINN surrogate reported an average $L^{2}$-error of 5.39\% when tested over fifty new parameter instances, a performance that we consider to be satisfactory. In particular, as it can be further appreciated from Figure \ref{fig:stokes}, the proposed MINN architecture succeeded in learning the main characteristics of the operator under study. Of note, we mention that, in this case, implementing the dense counterpart of architecture \ref{eq:stokesminn}, as we did back in Section \ref{sec:experiments1}), is computationally prohibitive. In fact, the latter would have $\approx300\cdot10^{6}$ degrees of freedom, which, for instance, our GPU cannot handle.
}

\begin{figure}
    \centering
    \includegraphics[width = 0.97\textwidth]{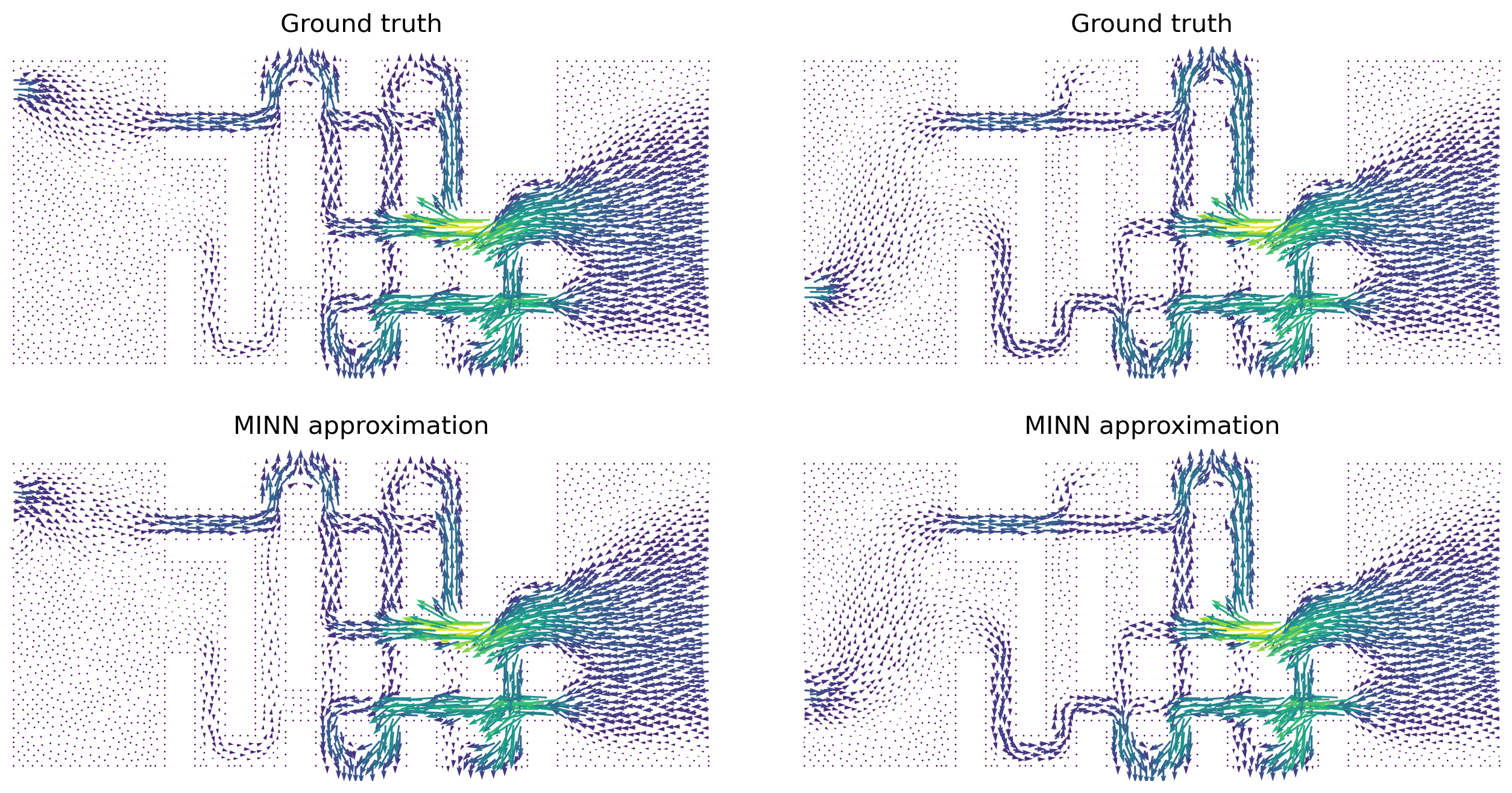}
    \caption{\newstuff{Ground truth vs MINN approximation for the Stokes flow example, Section \ref{subsec:stokes}. The solutions reported refer to two different configurations of the parameters (not seen during training), respectively on the left/right.}}
    \label{fig:stokes}
\end{figure}

\newstuff{\subsection{Brain damage recovery}
\label{subsec:brain}
As a second example we consider a problem concerning a nonlinear diffusion reaction of Fisher-Kolmogorov–Petrovsky–Piskunov (F-KPP) type \cite{fisher} in a two-dimensional domain representing a vertical section of a human brain across the sagittal plane. Specifically, we consider the following time-dependent PDE
\begin{equation}
\label{eq:brain}
\begin{cases}
    \displaystyle
    \frac{\partial u}{\partial t} - \frac{1}{10}\Delta u = 10^{3}u(1-u) & \text{in}\;\Omega\times(0,+\infty)\\
    \nabla u \cdot \mathbf{n}=0 & \text{in}\;\partial\Omega\times(0,+\infty)\\
    u(\mathbf{x},0)=\varphi(\mathbf{x})&\text{in}\;\Omega.
\end{cases}
\end{equation}
In the literature, researchers have used the F-KPP equation to model a multitude of biological phenomena, including cell proliferation and wound healing, see e.g. \cite{habbal,sherratt}. In this sense, we can think of \eqref{eq:brain} as an equation that models the recovery of a damaged human brain, even though much more sophisticated models would be required to properly describe such phenomenon \cite{ravasio}. 
\acapo
We thus interpret the solution to Equation \eqref{eq:brain}, $u:\Omega\times[0,+\infty)\to[0,1]$, as a map that describes the level of healthiness of the brain at a given location at a given time (0 = highly damaged, 1 = full health). Then, the dynamics induced by the F-KPP equation resembles that of a healing process: in fact, for any initial profile $\varphi:\Omega\to[0,1]$, later modeled as a random field, one has $u(\mathbf{x},t)\to 1$ for $t\to+\infty$, where the speed of such convergence depends both on $\varphi$ and on the point $\mathbf{x}$. In light of this, we introduce an additional map, $\tau:\Omega\to[0,+\infty)$, which we define as
\begin{equation}
\tau(\mathbf{x}):=\inf\left\{t\ge0\;|\;u(\mathbf{x},t)\ge0.9\right\}.    
\end{equation}
In other words, $\tau(\mathbf{x})$ corresponds to the amount of time required by the healing process in order to achieve a suitable recovery at the point $\mathbf{x}\in\Omega$. Our objective is to learn the operator
\begin{equation}
    \label{eq:brainoperator}
    \operator:\varphi\to\tau,
\end{equation}
which maps any given initial condition to the corresponding time-to-recovery map. To tackle this problem, we first discretize $\Omega$ through a triangular mesh of stepsize $h=0.0213$, over which we define the state space of continuous piecewise-linear finite elements, $V_{h}$, with $\text{dim}(V_{h})=2414$. Then, for any fixed initial condition $\varphi$, we solve \eqref{eq:brain} numerically by employing the Finite Element method in space and the backward Euler scheme in time ($\Delta t=10^{-4}$). In particular, in order to compute an approximation of $\tau$ at the mesh vertices, we evolve the PDE up until the first time at which all nodal values of the solution are above 0.9. In this way, for each $\varphi\in V_{h}$ we are able to compute a suitable $\tau\in V_{h}$, with $\tau\approx\operator(\varphi)$. Precisely, in order to guarantee that $\varphi$ always takes values in [0,1], and that its realizations have an underlying spatial correlation coherent with the geometry of $\Omega$, we cast our learning problem in a probabilistic setting where we endow the input space with the push-forward measure $\#\mathbb{P}$ induced by the map
\begin{equation*}
    g\to \frac{1}{2}\tanh\left(10g+5)\right)+\frac{1}{2}
\end{equation*}
where $\mathbb{P}$ is the probability law of a centered Gaussian random field $g$ defined over $\Omega$ with Covariance kernel
$$\text{Cov}(\x,\y):=e^{-100\geodistance(\x,\y)^{2}}$$
$\geodistance$ being the geodesic distance across $\Omega$.
\acapo
As model surrogate, we consider the MINN architecture reported below,
\begin{equation}
    \label{eq:brainminn}
    V_{h}\xrightarrow{\;r=0.05\;}V_{\frac{5}{2}h}\xrightarrow{\;r=0.1\;}V_{\frac{5}{2}h}\xrightarrow{\;r=0.05\;}V_{h}.
\end{equation}
which we train over 900 random snapshots and for a total of 300 epochs. As for our previous test cases, we use the 0.1-leakyReLU activation at the internal layers and we optimize the weights by minimizing the mean squared $L^{2}$-error: to do so, we rely on the L-BFGS optimizer, with default learning rate and no mini-batching. Note that, as for Section \ref{subsec:stokes}, here the supports of the mesh-informed layers are computed following the geodesic distance (stepsize of the auxiliary coarse mesh: $h^{*}=0.107$). However, differently from the previous test case, we only employ mesh-informed layers. This is because, intuitively, we expect the phenomenon under study to be mostly local in nature.
\\\\
After training, our architecture reported an average $L^{2}$-error of 7.40\% (computed on 100 randomly sampled unseen instances). Figure \ref{fig:brain} shows the initial condition of the brain for a given realization of the random field $\varphi$, together with its high-fidelity approximation for $\tau$ and the corresponding MINN prediction. Once again, we see a good agreement between the actual ground truth and the neural network output. For instance, the proposed MINN model succeeds in recognizing a fundamental feature of the healing process, that is: despite reporting an equivalent damage, two different regions of the domain may require completely different times for their recovery; in particular, those regions that are more isolated will take longer to heal.
\begin{figure}
    \centering
    \includegraphics[width=\textwidth]{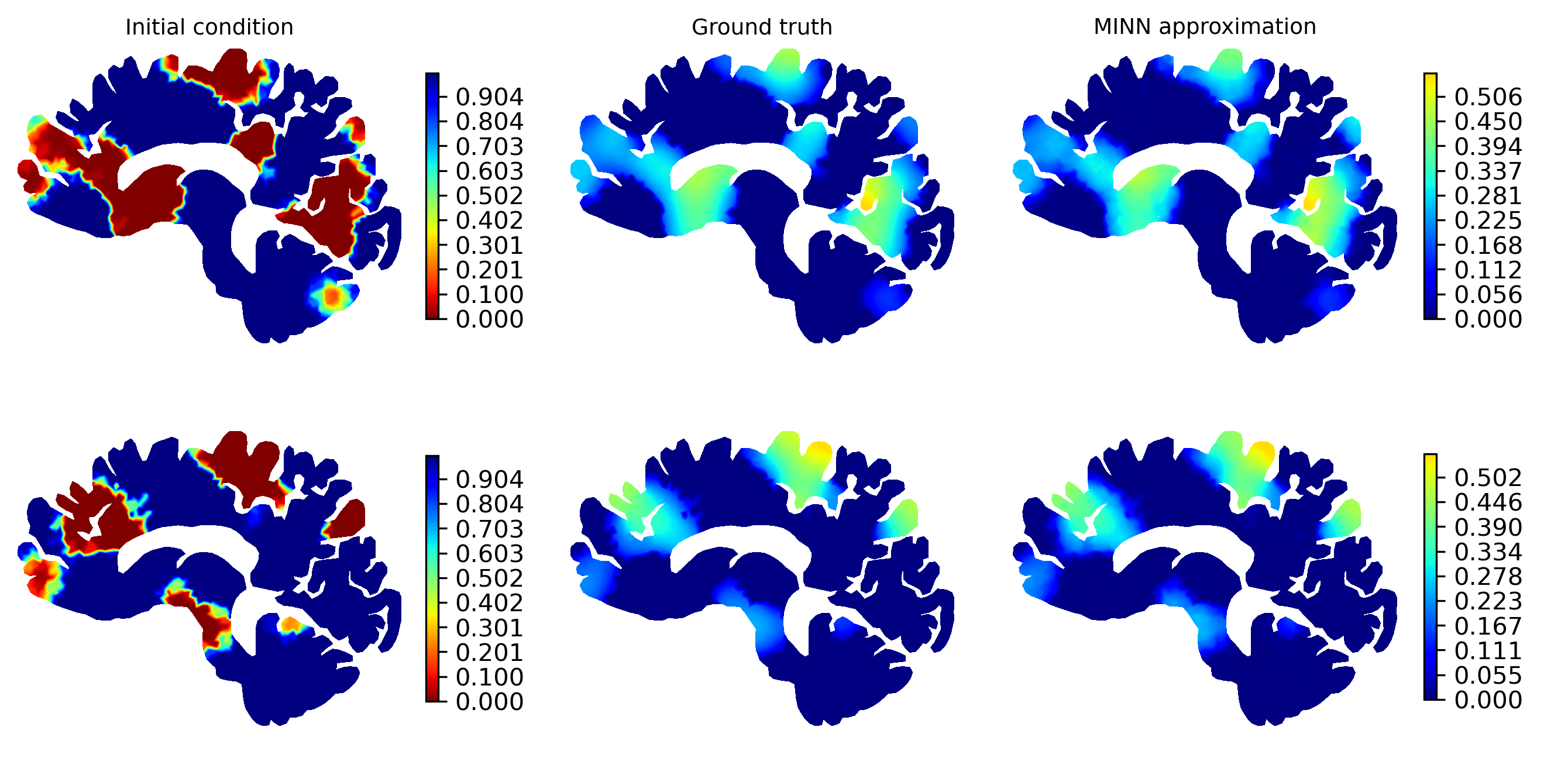}
    \caption{\newstuff{Initial condition (right), ground truth of the time-to-recovery map (center) and corresponding MINN approximation (right) for an unseen instance of the \textit{brain damage recovery} example, Section \ref{subsec:brain}. Note: a different colorbar is used for the initial condition and the time-to-recovery map, respectively.}}
    \label{fig:brain}
\end{figure}
}

\newstuff{\section{Numerical experiments: comparison with DeepONets and Fourier Neural Operators}
\label{sec:experiments4}
\renewcommand{\arraystretch}{1.5}
Our purpose for the current Section is to compare the performances of MINN architectures with those of two other popular Deep Learning approaches to operator learning, namely DeepONets and Fourier Neural Operators (FNO). To start with, we compare the three methodologies on a benchmark case study, where the performances of DeepONets and FNO have already been reported \cite{donfno}. We then move on to consider two additional problems which, despite their simplicity, we find to be of remarkable interest. These concern a dynamical system with chaotic trajectories, and an advection-dominated problem, two notoriously challenging scenarios when it comes to surrogate and reduced order modeling. Here, we shall exploit these problems to showcase the possible advantages that MINNs may offer with respect to the state-of-the-art.
\\\\
Here, for the sake of simplicity, we shall restrict our analysis to either 1D domains or 2D squares, which we always discretize through uniform grids. This is to ensure a fair comparison between MINNs and other approaches, such as FNOs, whose generalization to more complex geometries is still in the progress of being (see, e.g., the recently proposed Laplace Neural Operators, \cite{lno2,lno1}). We report our results in the subsections below. Before that, however, it may be worth recalling the fundamental ideas behind DeepONets and FNOs. 
\\\\
Let $\operator_{h}:\Theta\to V_{h}$ be our operator of interest, where $\Theta$ is some input space, while $V_{h}\cong\mathbb{R}^{N_{h}}$ is a given Finite Element space of dimension $N_{h}$, here consisting, for simplicity, of piecewise linear polynomials. 
As we mentioned in Section \ref{subsec:relationship}, DeepONets are grounded on a representation of the form $$(\mathcal{G}_{h}f)(\x)\approx\Psi(f)\cdot \phi(\x),$$
where $\Psi: \Theta\to \mathbb{R}^{m}$  and $\phi: \Omega\to\mathbb{R}^{m}$ are the branch and the trunk nets, respectively. Let now $\{\x_{j}\}_{j=1}^{N_{h}}$ be the nodes of the mesh at output, and let $\mathbf{V}\in\mathbb{R}^{N_{h}\times m}$ be the matrix
\begin{equation}
    \mathbf{V}:=\left[\begin{array}{ccccc}
    \phi_{1}(\x_{1}) &\;\;& \dots &\;\;&\phi_{m}(\x_{1})\\
    \vdots &&&&\vdots\\
    \phi_{1}(\x_{N_{h}}) &\;\;& \dots &\;\;&\phi_{m}(\x_{N_{h}})
    \end{array}\right],
\end{equation}
where $\phi_{i}(\x)$ denotes the $i$th component of $\phi(\x)$. We note that the map
\begin{equation}
    \Theta\ni f\to\mathbf{V}\cdot\Psi(f)\in\mathbb{R}^{N_{h}}\cong V_{h},
\end{equation}
formally returns the DeepONet approximation across the overall output mesh. In particular, this opens the possibility of directly replacing the trunk net with a suitable projection matrix $\mathbf{V}$. As a matter of fact, this is the strategy proposed by Lu et al. in \cite{donfno}, where the authors exploit Proper Orthogonal Decomposition (POD) \cite{negri} to compute the projection matrix $\mathbf{V}$ in an empirically optimal way. This technique, which the authors call POD-DeepONet, is usually better performing than the vanilla implementation of DeepONet, but it requires fixing the mesh discretization at output. Since the latter condition is always true within our framework, we shall restrict our attention to POD-DeepONet for our benchmark analysis. 
\\\\
For what concerns FNOs instead, the fundamental building block of these architectures is the \textit{Fourier layer}. Similarly to convolutional layers, these architectures are defined to accept inputs with multiple \textit{channels} or features. In our setting, for a given feature dimension $c$, a Fourier layer is a map $L:(V_{h})^{c}\to (V_{h})^{c}$ of the form
\newcommand{\vbold}{\boldsymbol{v}}
\begin{equation}
    L(\vbold)=\rho\left(\mathcal{F}^{-1}\left(\mathbf{R}\cdot\mathcal{F}(\vbold)\right)+W\vbold\right)
\end{equation}
where
\begin{itemize}
    \item[$\bullet$] $\mathcal{F}:(V_{h})^{c}\to\mathbb{R}^{m\times c}$ is the truncated Fourier transform, which takes the $c$-signals at input, computes the Fourier transform of each of them and only keeps the first $m$ coefficients;
    \item[$\bullet$] $\mathcal{F}^{-1}:\mathbb{R}^{m\times c}\to (V_{h})^{c}$ is the inverse Fourier transform, defined for each channel separately;
    \item[$\bullet$] $\mathbf{R}=(r_{i,j,k})_{i,j,k}\in\mathbb{R}^{m\times m\times c}$ is a learnable tensor that performs a linear transformation in the Fourier space. For any given input $\mathbf{A}=(a_{j,k})_{j,k}\in\mathbb{R}^{m\times c}$, the action of the latter is given by
$$\mathbf{R}\cdot\mathbf{A}:=\left(\sum_{j=1}^{m}r_{i,j,k}a_{j,k}\right)_{i,k}\in\mathbb{R}^{m\times c};$$
    \item[$\bullet$] $W:V_{h}\to V_{h}$ is a linear map whose action is given by 
    $$(W\vbold)(\x):=\mathbf{W}\cdot \vbold(\x),$$
    where $\mathbf{W}\in\mathbb{R}^{c\times c}$ is a learnable weight matrix;
    \item[$\bullet$] $\rho$ is the activation function.
\end{itemize}
In practice, when approximating a given operator $\operator_{h}:\Theta\to V_{h}$, the implementation of a complete FNO architecture is often of the form 
\begin{equation}
    P\circ L_{l}\circ\dots\circ L_{1}\circ E \circ \tilde{\Phi}:\quad\quad\Theta\to V_{h}\to(V_{h})^{c}\to\dots\to (V_{h})^{c}\to V_{h}.
\end{equation}
Here, $\Phi$ is a preliminary block, possibly consisting of dense layers, that maps the input form $\Theta$ to $V_{h}$. The FNO that exploits a lifting layer, $E$, to increase the number of channels, from $V_{h}$ to $(V_{h})^{c}$. Such lifting is typically performed by composing the signal at input with a shallow (dense) network, $e:\mathbb{R}\to\mathbb{R}^{c}$, so that $E:v(\x)\to e(v(\x))$. The model then implements $l$ Fourier layers, which constitute the core of the FNO architecture. Finally, a projection layer is used to map the output back from $(V_{h})^{c}$ to $V_{h}$. As for the lifting operator, this is achieved through the introduction of a suitable projection network, $p:\mathbb{R}\to\mathbb{R}^{c}$, acting node-wise, i.e. $P:\vbold(\x)\to p(\vbold(\x))$.
\subsection{Benchmark problem: diffusion equation with random coefficients}
\label{subsec:darcy}
To start, we consider a test case where the performances of DeepONets and FNOs have already been reported, see, respectively \cite{donfno} and \cite{fnos}. Let $\Omega=(0,1)^{2}$. We wish to learn the solution operator $\operator:L^{\infty}(\Omega)\to L^{2}(\Omega)$ of the boundary value problem below
\begin{equation}
    \label{eq:darcy}
    \begin{cases}
    -\nabla\cdot(a\nabla u) = 1 & \text{in}\;\Omega\\
    u = 0 & \text{on}\;\partial\Omega,
    \end{cases}
\end{equation}
where $\operator: a\to u$. As usual, we cast our learning problem in a probabilistic setting by endowing the input space, $L^{\infty}(\Omega)$, with a suitable probability distribution $\mathbb{P}$. In short, the latter is built so that $a$ satisfies
$$\mathbb{P}\left(a(\x)\in\{3,12\}\;\forall \x\in\Omega\right)=1,$$
meaning that the PDE is characterized by a random permeability field with piecewise-constant realizations. For a more accurate description about the construction of such $\mathbb{P}$, we refer to \cite{fnos,donfno}.
\acapo
Following the same lines of \cite{donfno}, we discretize the spatial domain with a uniform 29x29 grid, which we can equivalently think of as a triangular rectangular mesh, $\mesh=\{\x_{i}\}_{i=1}^{841}$, of uniform stepsize $h=\sqrt{2}/28\approx0.05.$ Both for training and testing, we refer to the dataset made available by the authors in \cite{donfno}, which consists of 1000 training snapshots and 200 testing instances, respectively. As detailed in \cite{fnos}, the latter were obtained by repeatedly solving Problem \eqref{eq:darcy} via finite-differences. Consequently, for each random realization of the permeability field, we have access to the nodal values $\mathbf{a}=[a(\x)]_{\x\in\mesh}\in\mathbb{R}^{N_{h}}$ of the random field, and the corresponding approximations of the PDE solution, $\mathbf{u}\approx[u(\x)]_{\x\in\mesh}\in\mathbb{R}^{N_{h}}$, where $N_{h}=841.$ In this sense, we can think of the discrete operator as $\operator_{h}:V_{h}\to V_{h}$, with $V_{h}=X_{h}^{1}$, even though it would be more natural to have $X_{h}^{0}$ as input space to account for the discontinuities in $a$.
\acapo
It is also in light of these considerations that we propose to learn to operator with the MINN architecture below,
\begin{equation}
    \label{eq:darcyminn}
    X_{h}^{1}
    \xrightarrow{\;r=0.1\;}
    X_{4h}^{0}
    \xrightarrow{\;r=0.4\;}
    X_{4h}^{0}
    \xrightarrow{\;r=0.2\;}
    X_{h}^{1}    
    \xrightarrow{\;r=0.2\;}
    X_{h}^{1}
    \xrightarrow{\;r=0.1\;}
    X_{h}^{1}
    \xrightarrow{\;r=0.2\;}
    X_{h}^{1}.
\end{equation}
The idea is that the first two layers should act as a preprocessing of the input. In particular, aside from the dimensionality reduction, they serve the additional purpose of recasting the original signal over the space of P0 Finite Elements, which we find to be better suited for representing permeability fields.
\acapo
This time, we employ the 0.3-leakyReLU activation for all the layers, including the last one. This allows us to enforce, at least 
in a relaxed fashion, the fact that any PDE solution to \eqref{eq:darcy} should be nonnegative. We train our model for 300 epochs by minimizing the mean squared $L^{2}$-error through the L-BFGS optimizer (no batching, default learning rate), which on our GPU takes only 1 minute an 56 seconds.
\\\\
After training, our model reports an average $L^{2}$-error of 2.78\%, which we find to be satisfactory as it compares very well with both DeepONets and FNOs (cf. Table \ref{tab:luflow}). In general, as it can be observed from Figure \ref{fig:luflow}, the proposed MINN architecture can reproduce the overall behavior of the PDE solutions fairly well, but it fails in capturing some of their local properties. Considering that our model is twice as accurate over the training snapshots than it is over the test data (average relative error: 1.38\% vs 2.78\%), this might be due to a reduced amount of training instances.
\acapo
Nonetheless, our performances remain comparable with those achieved by the state-of-art. Here, DeepONets attain the best accuracy thanks to their direct usage of the POD projection, which is known to be particularly effective for elliptic problems \cite{negri}. Indeed, as shown in \cite{donfno}, the same approach yields an average $L^{2}$-error of 2.91\% if one replaces the POD basis with a classical trunk-net architecture. FNOs, instead, report the worst performance, by their accuracy can be easily increased if one exploits suitable strategies such as \textit{output normalization}. The latter consists in the construction of a surrogate model of the form
$$\operator_{h}(a)\approx \sigma\tilde{\Phi}(a)+\bar{u},$$
where $\tilde{\Phi}$ is an FNO architecture to be learned, while $$\bar{u}=\bar{u}(\x)\approx\mathbb{E}[u]\quad\quad\text{and}\quad\quad\sigma^{2}=\sigma^{2}(\x)\approx\mathbb{E}|u-\mathbb{E}[u]|^{2}$$ are the pointwise average and variance of the solution field, respectively (both to be estimated directly from the training data). Then, this trick allows one to obtain a better FNO surrogate with a relative error of 2.41\% \cite{donfno}.
\begin{table}
    \centering
    \begin{tabular}{ l | l  l  l } 
    \hline
  & \textbf{MINN} & \textbf{FNO} & \textbf{POD-DeepONet} \\
  \hline\hline\rowcolor{Gray}
 $L^{2}$-relative error & 2.78\% & 4.83\% & 2.32\%\\\rowcolor{Gray2}
Layers & 6 Mesh-informed & 4 Fourier & 2 Dense\\\rowcolor{Gray2}
 &&2 Feature maps&2 Convolutional\\\rowcolor{Gray}
 Hyperparameters  &  $h_{\text{coarse}}=4h$ & $c=32$ & $m_{\text{POD}}=115$\\\rowcolor{Gray}
 &&$m=12$&\\
 \hline
    \end{tabular}
    \caption{\label{tab:luflow} 
\newstuff{Comparison of Mesh-Informed Neural Networks, DeepONets and Fourier Neural Operators (FNO), for the diffusion equation example, Section \ref{subsec:darcy}. All errors correspond to the average performance reported across the test set. For POD-DeepONet, the number of layers refers to the branch net only. $h$ = stepsize of the reference mesh, $h_{\text{coarse}}$ = stepsize of the coarser mesh, $c$ = no. of features for the Fourier layers, $m$ = no. of Fourier modes, $m_{\text{POD}}$ = no. of POD basis functions.}}
\end{table}
\begin{figure}
    \centering
    \includegraphics[width=0.9\textwidth]{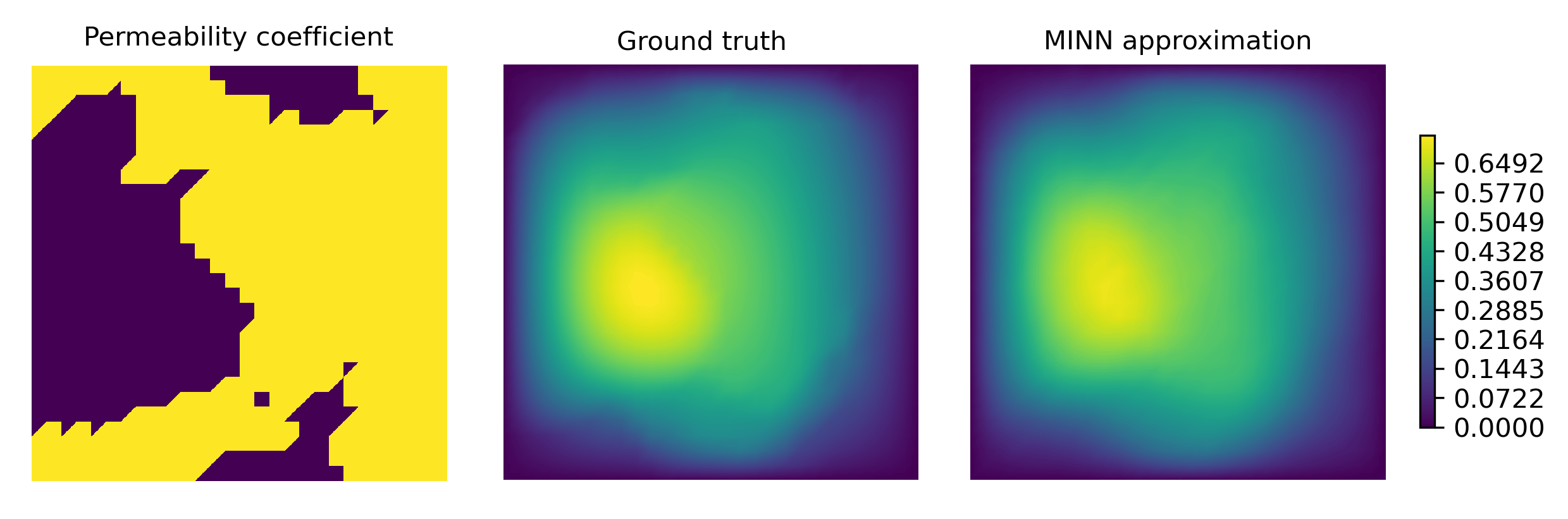}
    \caption{\newstuff{Input field, ground truth and MINN approximation for a test instance of Problem \eqref{eq:darcy}, Section \ref{subsec:darcy}.}}
    \label{fig:luflow}
\end{figure}
\subsection{Dealing with chaotic trajectories: the Kuramoto–Sivashinsky equation}
As a second example, we test the three methodologies in the presence of chaotic behaviors, here arising from the Kuramoto-Sivashinsky equation, a time-dependent nonlinear PDE that was first introduced to model thermal instabilities and flames propagation \cite{kuramoto}. More precisely, we consider the periodic boundary value problem below,
\label{subsec:kuramoto}
\begin{equation}
    \begin{cases}
    \label{eq:kuramoto}\displaystyle 
    \;\frac{\partial u}{\partial t} = 
    -\nu\frac{\partial^{4}u}{\partial x^{4}}
    -\frac{\partial^{2} u}{\partial x^{2}} - u\frac{\partial u}{\partial x} & \text{in}\;\mathbb{R}\times(0,T]\vspace{0.5em}\;
    \\\vspace{0.5em}\;
    u(x,t)=u(x+\ell,t) &  (x,t)\in\mathbb{R}\times(0,T]
    \\\vspace{0.5em}\;
    u(x,0)=u_{0}(x) &  x\in\mathbb{R}
    \end{cases}
\end{equation}
where $\ell=T=100$ and $u_{0}(x):=\pi+\cos(2\pi x/\ell)+0.1\cos(4\pi x/\ell)$ are given, while $\nu\in[1,3.5]$ is a parameter. We wish to learn the operator $\operator$ that maps $\nu$ onto the corresponding PDE solution $u=u(x,t)$. To this end, it is convenient to define the spacetime domain $\Omega:=(0,\ell)\times(0,T)$, which we equip with a uniform grid of dimension 100x100. As a high-fidelity reference, we then consider a spectral method combined with a modified Crank-Nicolson scheme for time integration ($\Delta t = 5\cdot10^{-4}$: trajectories are later subsampled to fit the uniform grid over $\Omega$). This allows us to define the discrete operator as $$\operator_{h}:\Theta\to V_{h},$$ with where $\Theta:=[1,3.5]$ and $V_{h}:=X^{1}_{h}(\Omega)$, so that $\dim(X^{1}_{h})=N_{h}=n_{h}^{2}=10'000.$ We use the numerical solver to compute a total of $N=1000$ PDE solutions, $\{(\nu_{i},u_{h}^{i})\}_{i=1}^{N}$, $N_{\text{train}}=500$ to be used for training and $N_{\text{test}}=500$ for testing.
\\\\
We train all the models according to the loss function below
$$\mathscr{L}(\Phi):=\frac{1}{N_{\text{train}}}\sum_{i=1}^{N_{\text{train}}}\left(\frac{1}{n_{h}^{2}}\sum_{j,k=1}^{n_{h}}|u_{h}^{i}(x_{j},t_{k})-\Phi(\nu_{i})(x_{j},t_{k})|^{2}\right)$$
which we minimize iteratively, for a total of 500 epochs, using the L-BFGS optimizer. Similarly, in order to emphasize the difference between the space and the time dimension, we evaluate the accuracy of the models in terms of the relative error below,
\begin{multline}
    \label{eq:l1l2err}
    \mathscr{E}(\Phi):=\\=\frac{1}{N_{\text{test}}}\sum_{i=N_{\text{test}}}^{N}\left(\frac{\frac{1}{n_{h}}\sum_{k=1}^{n_{h}}\sqrt{\frac{1}{n_{h}}\sum_{j=1}^{n_{h}}|u_{h}^{i}(x_{j},t_{k})-\Phi(\nu_{i})(x_{j},t_{k})|^{2}}}{\frac{1}{n_{h}}\sum_{k=1}^{n_{h}}\sqrt{\frac{1}{n_{h}}\sum_{j=1}^{n_{h}}|u_{h}^{i}(x_{j},t_{k})|^{2}}}\right)\approx\\\\
    \approx \frac{1}{N_{\text{test}}}\sum_{i=N_{\text{test}}}^{N}\left(
    \frac{\|u_{h}^{i}-\Phi(\nu_{i})\|_{L^{1}((0,T);\; L^{2}(0,L))}}{\|u_{h}^{i}\|_{L^{1}((0,T);\;L^{2}(0,L))}}\right),
\end{multline}
where norms are intended in the Bochner sense \cite{evans}. For the implementation of the three approaches, we proceed as follows:
\begin{itemize}
    \item [i)] POD-DeepONet: we exploit the training data to construct a POD basis $\mathbf{V}$ consisting of $m=200$ modes, as those should be sufficient for capturing most of the variability in the solution space. 
    We then construct the branch net as a classical DNN from $\mathbb{R}\to\mathbb{R}^{m}$, with 3 hidden layers of width 500, each implementing the 0.1-leakyReLU activation;\vspace{0.5em}
    \item [ii)] FNO: we use a combined architecture with a dense block at the beginning and three Fourier layers at the end. The dense block consists of 2 hidden layers of width 500, and an output layer with $N_{h}$ neurons, all complemented with the 0.1-leakyReLU activation. Then, the model is followed by three Fourier layers, which, for simplicity, are identical in structure. In particular, following the same rule of thumb proposed by the authors \cite{fnos}, we implement a Fourier block with $c=32$ features and $m=12$ Fourier modes per layer. Once again, all layers (except for the last one) use the 0.1-leakyReLU as nonlinearity;\vspace{0.5em}
    \item [iii)] MINN: we propose the following architecture, obtained through a combination of dense and mesh-informed layers
    \begin{equation*}
    \mathbb{R}\to\mathbb{R}^{50}\to V_{3h}\xrightarrow{\;r=20\;} V_{h}\xrightarrow{\;r=6\;}V_{h},\end{equation*}
    where we recall that the stepsize of the (finer) spacetime mesh is $h=\sqrt{2}$. This time, however, we do not use the Euclidean metric for the computation of the supports, but we rely on a different distance function that can account for the periodicity of the spatial component. In fact, Equation \eqref{eq:kuramoto} is better understood in the periodic domain $\tilde{\Omega}:=(\mathbb{R}/\ell\mathbb{Z})\times[0,T]$, than it is over $\Omega$. Here, $\mathbb{R}/\ell\mathbb{Z}$ is the quotient of $\mathbb{R}$ with respect to the equivalence relation below,
    $$x\sim x'\iff (x-x')\ell^{-1} \in \mathbb{Z}.$$  
    To account for this, we compute the supports of the mesh-informed layers according to the following distance function $\geodistance:\Omega\times\Omega\to[0,+\infty)$,
    $$\geodistance((x,t),(x',t')):=\sqrt{\left(\frac{\ell}{2}-\left|\frac{\ell}{2}-|x-x'|\right|\right)^{2}+(t-t')^{2}}.$$
    This ensures the wished behavior over the unwrapped domain $\Omega$. In fact, for instance, one has $$\geodistance((1,0),(3,0))=\geodistance((1,0),(99,0)),$$
    as we recall that $\ell=100.$
\end{itemize}
\;\\
After training, all model surrogates report relative errors below 10\%, with MINNs achieving the best performance, cf. Table \ref{tab:ksh}. When tested over parametric instances outside of the training set, the three approaches propose similar but different predictions, cf. Figure \ref{fig:ksh}. The FNO-surrogate manages to capture the macroscopic features of the solution, but returns a much noiser output. Conversely, approximation proposed by POD-DeepONet is cleaner and better resembles the original ground truth. Still, even the latter model fails to capture some of the local features of the PDE solution, which, in contrast, the MINN architecture is able to recover. The interested reader may also find additional insights in Appendix A.
\acapo
Nonetheless, we must acknowledge that our analysis might be affected by the several design choices that we had to make for the architectures. In fact, even though we did our best in implementing and tuning the hyperparameters for the three approaches, we cannot consider our results to be universal.
\begin{table}
    \centering
    \begin{tabular}{ l | l  l  l } 
    \hline
  & \textbf{MINN} & \textbf{FNO} & \textbf{POD-DeepONet} \\
  \hline\hline\rowcolor{Gray}
 $(L^{1}\;$of $L^{2})$-relative error & 5.87\% & 9.41\% & 6.97\%\\\rowcolor{Gray2}
Layers & 2 Mesh-informed & 3 Fourier & 4 Dense\\\rowcolor{Gray2}
 &2 Dense & 3 Dense & \\\rowcolor{Gray2}
 &&2 Feature maps&\\\rowcolor{Gray}
 Hyperparameters  &  $h_{\text{coarse}}=3h$ & $c=32$ & $m_{\text{POD}}=200$\\\rowcolor{Gray}
 &&$m=12$&\\
 \hline
    \end{tabular}
    \caption{\label{tab:ksh} 
\newstuff{Models comparison for the Kuramoto Sivashinsky example, Section \ref{subsec:kuramoto}. Error are computed according to Eq. \eqref{eq:l1l2err}. Table entries for \textit{Layers} and \textit{Hyperparameters} read as in Table \ref{tab:luflow}}.}
\end{table}
\begin{figure}
    \centering
    \includegraphics[width=\textwidth]{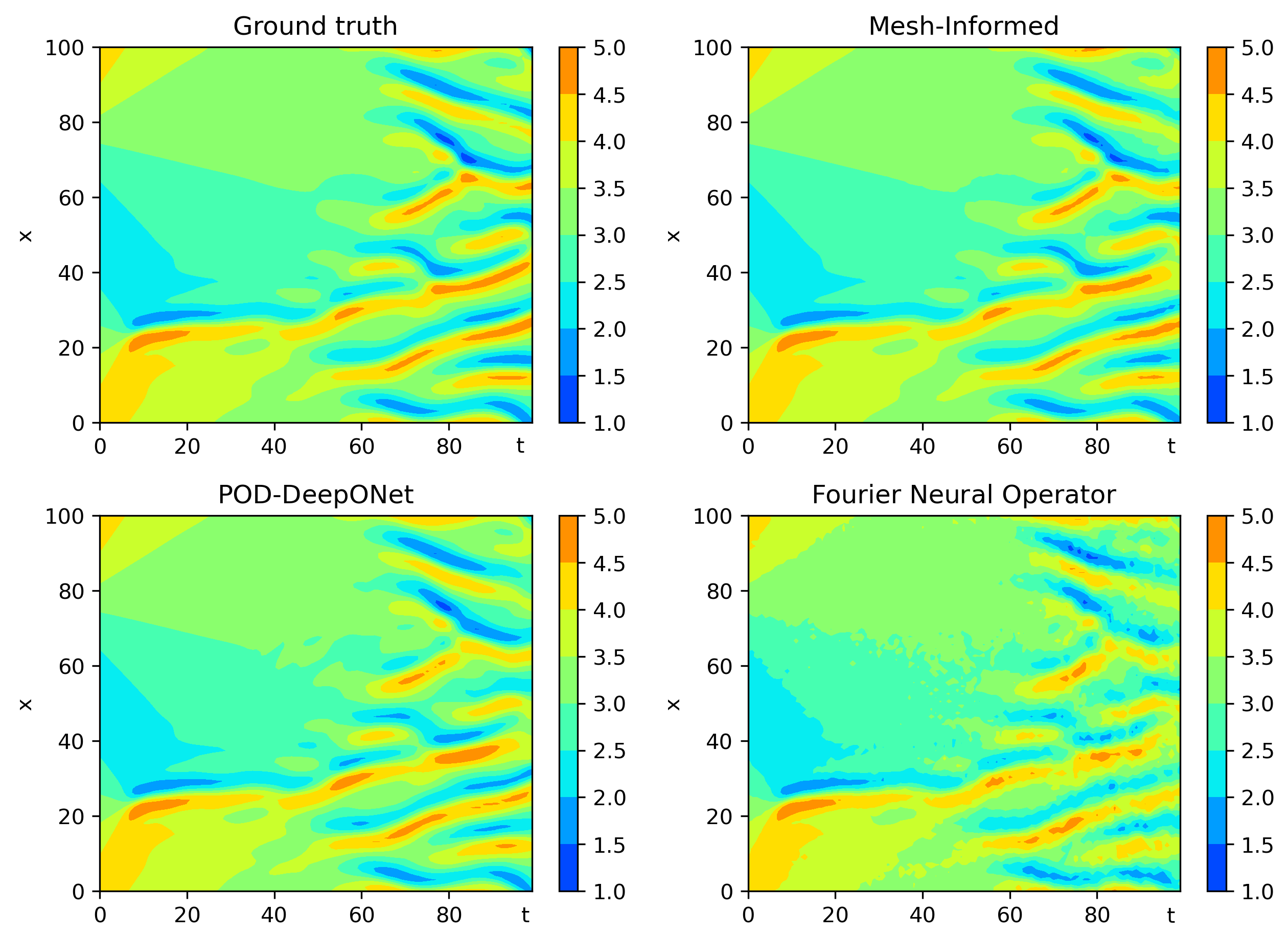}
    \caption{\newstuff{Ground truth and corresponding approximations via MINN, POD-DeepONet and FNO, for a test instance of the Kuramoto-Sivashinsky example, Section \ref{subsec:kuramoto}. The solutions are plotted over the spacetime domain $\Omega=(0,\ell)\times(0,T).$}}
    \label{fig:ksh}
\end{figure}
\subsection{Dealing with the Kolmogorov barrier: the advection equation}
\label{subsec:kolmogorov}
As a last example, we consider a situation in which the problem at hand exhibits a slow decay of the \textit{Kolmogorov n-width}. The latter is a quantity of particular interest in the areas of surrogate and Reduced Order Modeling, as it measures to which extent a given phenomenon can be described via linear superposition of suitable modes. More precisely, given a compact subset $S$ of a normed state space $(V,\|\cdot\|)$, the Kolmogorov $n$-width of $S$ is defined as
$$d_{n}(S):=\inf_{\substack{H\subseteq V\\\dim(H)=n}}\;\sup_{u\in S}\;\inf_{v\in H}\|u-v\|.$$
For instance, if $d_{n}(S)<\varepsilon$ for some $n\in\mathbb{N}$, then there will be $n$ modes, $v_{1},\dots,v_{n}\in V$, capable of representing all the elements in $S$ with an error that is smaller than $\varepsilon$. Viceversa, if $d_{n}(S)$ decays slowly with $n$, then one is forced to consider a larger number of modes to obtain a reasonable approximation. As we shall discuss in a moment, this fact can have a huge impact for operator learning problems, especially for those approaches based on linear projection techniques, such as DeepONets.
\acapo
To see this, consider a continuous operator $\operator:\Theta\to L^{2}(\Omega)$, with $\Theta\subset\mathbb{R}^{p}$ a compact subset and $\Omega$ a bounded domain. Let $S:=\operator(\Theta)$ be the image of $\Theta$ through the operator $\operator.$ Then, it is straightforward to see that
\begin{equation}
\label{eq:iff}
d_{n}(S)<\varepsilon\iff\end{equation}
$$
\exists \{v_{i}\}_{i=1}^{n}\subset L^{2}(\Omega),\;\{\phi_{i}\}_{i=1}^{n}\subset\mathcal{C}(\Theta)\;\;\text{s.t.}\;\;\sup_{\mup\in\Theta}\left\|\operator(\mup)-\sum_{i=1}^{n}\phi_{i}(\mup)v_{i}\right\|_{L^{2}(\Omega)}<\varepsilon,$$
furthermore, since the projection coefficients are optimal on any given (orthonormal) basis, it is not restrictive to set $\phi_{i}(\mup):=\langle \operator(\mup),v_{i}\rangle_{L^{2}(\Omega)}.$ The equivalence in \eqref{eq:iff} shows that the effectiveness of a separation of variables approach, such as the one adopted by DeepONets and POD-DeepONets, is confined by the behavior of the Kolmogorov $n$-width. In particular, these approaches may encounter some difficulties if $d_{n}(S)$ happens to decay slowly.
\\\\
In light of these considerations, we propose a final case study based on the advection equation, which, despite its simplicity, constitutes a prototypical example of slow decay in the Kolmogorov $n$-width. More precisely, let us consider the evolution equation below
\begin{equation}
\label{eq:adve}
\begin{cases}
    \displaystyle 
    \frac{\partial u}{\partial t} +
    c\frac{\partial u}{\partial x} = 0 & \text{in}\;\Omega\times(0,T]\\\\
    u(0,t)=u_{0}(-ct) & t\in(0,T]\\
    u(x,0)=u_{0}(x) & x\in\Omega
\end{cases}
\end{equation}
where $\Omega=(0,1)$, $T=1$ and
$$u_{0}(x):=\frac{1}{\sqrt{2\pi\sigma^{2}}}e^{-\frac{1}{2\sigma^{2}}x^{2}},$$with $\sigma=0.005$, 
are given, while the wave velocity, $c$, is regarded as a model parameter. Our goal is to approximate the operator $$\operator:c\to u(\cdot, T),$$ that maps the wave speed to the terminal state of the system. Equivalently, we may write
$(\operator c)(x)=u(x-c)$,
as the solution to \eqref{eq:adve} is known in closed form. We exploit the latter formula to generate a collection of 1000 random snapshots, where the input parameter is sampled uniformly from $\Theta:=[0.05,\;0.95]$. We keep 750 of such snapshots for training a leave the remaining 250 for testing. For the implementation of the three Deep Learning approaches we proceed as follows:
\begin{itemize}
    \item [i)] POD-DeepONet: we set up two possible architectures. For the first one, we use 50 POD modes and a DNN with 2 hidden layers of width 100 as branch net. As a second architecture, instead, we consider a similar model where the number of POD basis is increased to 200;
    \\
    \item [ii)] FNO: we construct a model by concatenating a shallow network of width 100, which maps $\mathbb{R}\to\mathbb{R}^{N_{h}}$, with a Fourier block composed of two Fourier layers. We design the latter along the same lines adopted by the authors in \cite{fnos}, i.e. by letting the number of Fourier modes be equal to 16 and by introducing a total of 64 hidden features.
    \\
    \item [iii)] MINN: we use the hybrid architecture below
$$\mathbb{R}\to\mathbb{R}^{100}\to\mathbb{R}^{100}\to V_{20h}\xrightarrow{\;r=0.1\;} V_{4h}\xrightarrow{\;r=0.01\;} V_{h}.$$    
\end{itemize}
As usual, we equip the internal layers of all our architectures with the 0.1-leakyReLU activation. Similarly, we train the three models by minimizing the mean squared $L^{2}$-error, as we did for all our previous experiments. This time, we rely on the Adam optimizer for the minimization of the loss function, which we run for a total of 50'000 epochs, starting with a learning rate of $10^{-3}$ and halving it every 10'000 iterations. Here, in fact, for all of the approaches we found the Adam optimizer to yield better results with respect to L-BFGS.
\\\\
Results are in Table \ref{tab:adve}, Figure \ref{fig:adve} and Appendix A. Unsurprisingly, the two POD-DeepONets achieve the highest errors, which is to be expected considering the nature of the problem itself. In the first case, this is due to the poorness of the POD basis, which, with as little as 50 modes, is unable to capture the overall phenomenon. In the second case, instead, this is caused by large number of POD coefficients to be learned by the branch net. In fact, in our analysis, the largest POD basis (200 modes) reached an average projection error below 0.01\%: therefore, here, the branch net is the only one responsible for the bad performance of the model. 
\acapo 
Conversely, things improve a lot if we move to fully nonlinear methods, such as FNOs and MINNs. For instance, FNO is at least four times more accurate than POD-DeepONet, with an average relative error of 11.08\%. Still, its performance is not comparable with the one achieved by our MINN surrogate (relative error: 4.18\%). This difference may be further appreciated by the plots reported in Figure \ref{fig:adve}: while all the models predict the location of the wave correctly, only the MINN surrogate manages to capture both the shape and the magnitude of the signal; in contrast the approximations proposed by FNO and POD-DeepONet present spurious oscillations and are either noisy or off-scale. This goes to show that even an apparently simple phenomenon can give rise to highly nontrivial complications. 
\begin{table}
    \centering
    \begin{tabular}{ l | l  l  l} 
    \hline
  & \textbf{MINN} & \textbf{FNO} & \textbf{POD-DeepONet} \\
  \hline\hline\rowcolor{Gray}
 $L^{2}$-relative error & 4.18\% & 11.08\% & 42.63\% (50 basis)\\\rowcolor{Gray}
 &  &  & 45.87\% (200 basis)\\\rowcolor{Gray2}
Layers & 2 Mesh-informed & 3 Fourier & 4 Dense\\\rowcolor{Gray2}
 &3 Dense & 3 Dense & \\\rowcolor{Gray2}
 &&2 Feature maps&\\\rowcolor{Gray}
 Hyperparameters  &  $h_{\text{coarse}}=20h$ & $c=64$ & $m_{\text{POD}}=50-200$\\\rowcolor{Gray}
 &$h_{\text{inter}}=4h$&$m=16$&\\
 \hline
    \end{tabular}
    \caption{\label{tab:adve} 
\newstuff{Test errors for the advection equation, Section \ref{subsec:kolmogorov}. Table entries as in Table \ref{tab:luflow}. Here, $h_{\text{inter}}$ = intermediate resolution.}}
\end{table}
\begin{figure}
    \centering
    \includegraphics[width=\textwidth]{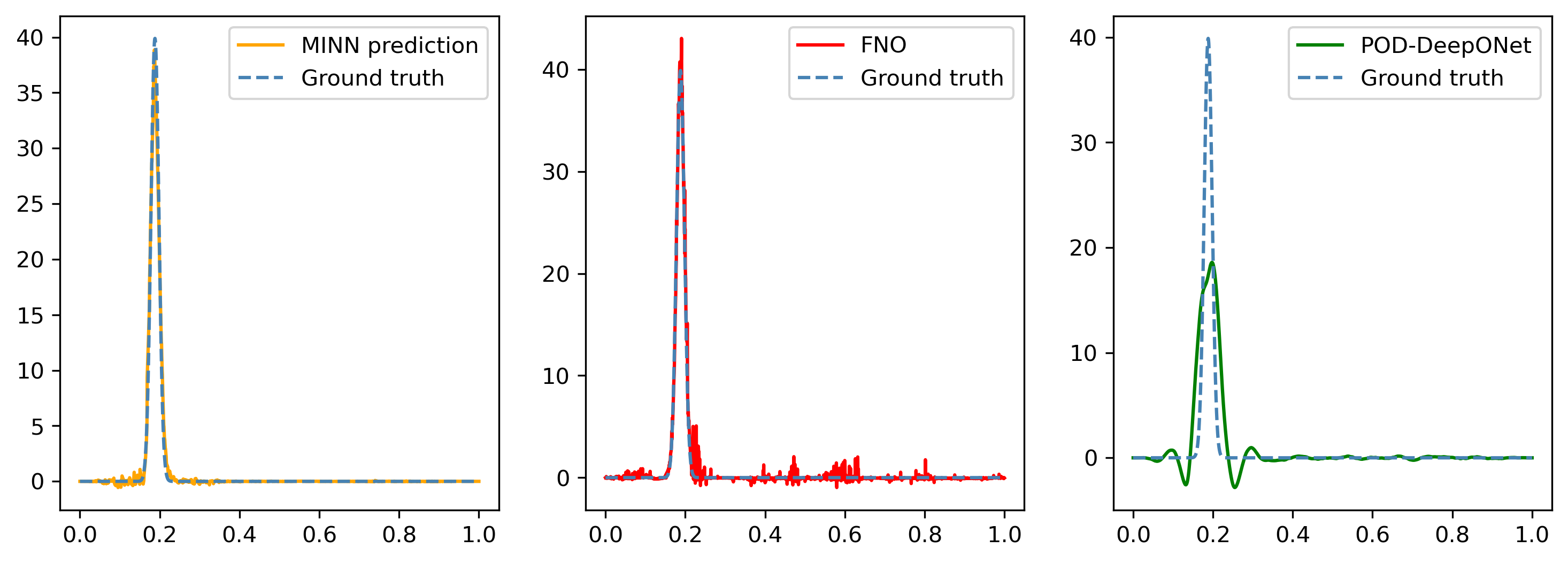}
    \caption{\newstuff{Model predictions of MINN (left), FNO (center) and POD-DeepONet (right) for an unobserved parametric instance of the advection equation problem, Section \ref{subsec:kolmogorov}. For POD-DeepONet, only the output of the best model is reported (50 basis).}}
    \label{fig:adve}
\end{figure}
}

\section{An application to Uncertainty Quantification}
\label{sec:uq}
We finally consider an application to Uncertainty Quantification (UQ) involving a partial differential equation. UQ is an essential aspect of robust modelling, which often involves expensive numerical and statistical routines. In this Section, we provide an example on how MINNs can alleviate these costs by serving as model surrogates in the computational pipeline. In particular, starting from a suitable PDE model, we address a problem concerning oxygen transfer in biological tissues. 

\subsection{Model description}
Oxygen is a fundamental constituent of most biological processes. In humans, oxygen is delivered by the circulatory system from the lungs to the rest of the body. At the small scales, cells receive oxygen from the vascular network of capillaries that spread all over the body. An efficient oxygen transfer is fundamental to ensure a healthy micro-environment and abnormal values in oxygen concentration are often associated to pathological scenarios. \newstuff{In particular, \textit{hypoxia}, that is the shortage of oxygen supplies,} plays an important role in the development and treatment of tumors. It has been shown that hypoxic tissue opposes a resistance to chemotherapy and radiotherapy \cite{cattaneo,possenti}. These issues are caused by perturbed properties of the tumor blood vessels in terms of morphology and phenotype. Here, we aim at developing a methodology to assess the role of vascular morphology on tissue hypoxia. More precisely, we wish to address the following question: how does the topology of the vascular network relate to the size of the tissue under hypoxia? We answer this question in the simplified setting that we describe below.
\begin{figure}
\label{fig:vessels}
\centering
\includegraphics[width=0.7\textwidth]{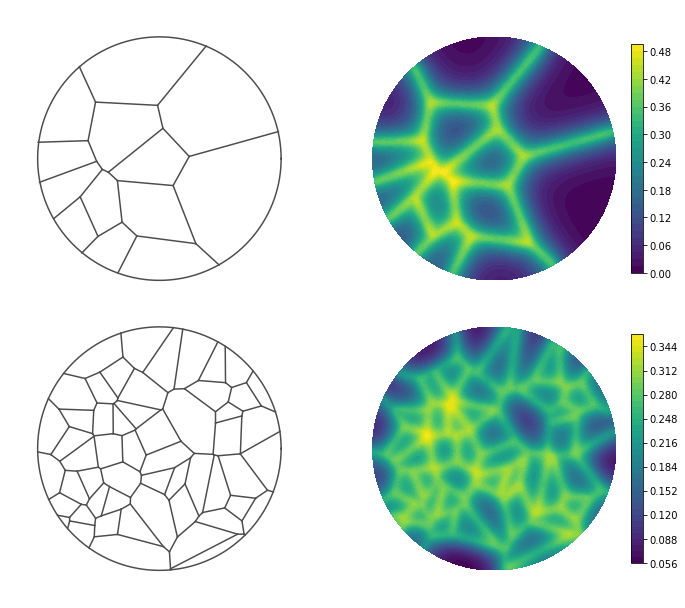}
\caption{Forward UQ problem (Section \ref{sec:uq}). Topology of the microvascular network $\Lambda$ (left) and corresponding oxygen distribution in the tissue $u=u_{\Lambda}$ (right). The top and the bottom rows corresponds respectively to a poorly and a highly vascularized tissue (resp. $\lambda = 1$ and $\lambda = 3$). Globally, the two networks provide the same amount of oxygen (cf. Equation \ref{eq:oxy}), but their topology significantly affects the values of $u$ in the tissue. In the first case (top row), nearly 31\% of the tissue has an oxygen level below the threshold value $u_{*}:=0.1$. Conversely, only 3\% of the tissue reports a low oxygen concentration in the second example.}
\end{figure}
\\\\
Within an idealized setting, we consider a portion of a vascularized tissue $\Omega:=\{\x\in\mathbb{R}^{2}:|x|<1\}$. Let $\Lambda\subset\overline{\Omega}$ be a graph representing the vascular network of capillaries (cf. Figure 7) and let $u:\Omega\to[0,1]$ be the oxygen concentration in the tissue, normalized to the unit value. We model the oxygen transfer from the network to the tissue with the following equations,
$$\begin{cases}
    \label{eq:microstrong}
    -\alpha\Delta u + u = (1-u)\delta_{\Lambda} & \text{in }\Omega\\
    -\alpha\nabla u \cdot \mathbf{n} = \beta u & \text{on }\partial\Omega
\end{cases}$$
where $\alpha=0.1$ and $\beta=0.01$ are respectively a fixed diffusion and resistance coefficient, while $\delta_{\Lambda}$ is the unique singular measure for which
$$\int_{\Omega}v(\x)\delta_{\Lambda}(d\x) = \frac{1}{|\Lambda|}\int_{\Lambda}v(\s)d\s$$
for all $v\in\mathcal{C}(\overline{\Omega})$. Here, we denote by $|\Lambda|:=\int_{\Lambda}1d\s$ the total length of the vascular graph. 
The first equation in \eqref{eq:microstrong} describes the diffusion and consumption of oxygen, balanced accordingly to the amount released from the vascular network on the right hand side. Finally, the model is closed using resistance boundary conditions of Robin type.

We understand \eqref{eq:microstrong} in the weak sense, meaning that define $u=u_{\Lambda}$ as the unique solution to the problem below
\begin{multline}
    \label{eq:oxy}
    \int_{\Omega}\alpha\nabla u(\x) \cdot \nabla v(\x)d\x + \int_{\Omega}u(\x)v(\x)d\x
    + \int_{\partial\Omega}\beta u(\s)v(\s)d\s = \\ = \frac{1}{|\Lambda|}\int_{\Lambda}(1-u(\s))v(\s)d\s
\end{multline}

\noindent\newline where the above is to be satisfied for all $v\in H^{1}(\Omega)$.

\subsection{Uncertainty quantification setting}
As we mentioned previously, we are interested in the relationship between $\Lambda$ and $u$. To this end, we introduce the parameter space
$$\Theta:=\{\Lambda\subset\overline{\Omega}:\;\Lambda\text{ is the union of finitely many segments}\}$$
which consists of all vascular networks. Note that, due to the normalizing factor $1/|\Lambda|$ in \eqref{eq:oxy}, all the vascular networks actually provide the same global amount of oxygen. However, as we will see later on, only those vascular graphs that are sufficiently spread across the domain can ensure a proper oxygen supply to the whole tissue (cf. Figure 7). In other words, we explore the influence of the distribution of the network, rather than its density, on the oxygen level.

The next subsection is devoted to prescribing a suitable discretization of \eqref{eq:oxy} to work with, and to introduce a class of probability measures $\{\mathbb{P}_{\lambda}\}_{\lambda}$ defined over $\Theta$. The idea is the following. We will use a macro-scale parameter $\lambda$ to describe the general perfusion of the tissue. Higher values of $\lambda$ will correspond to a highly vascularized tissue. This means that the topology of the vascular network will still be uncertain, but the corresponding probability distribution $\mathbb{P}_{\lambda}$ will favor dense graphs. Conversely, lower values of $\lambda$ will describe scenarios where capillaries are more sparse (see Figure 7, top vs bottom row). This will then bring us to consider the family of random variables
\begin{equation}
    \label{eq:qoi}
    Q_{\lambda}:=\frac{1}{|\Omega|}|\{u_{\Lambda} < 0.1\}| \text{ with } \Lambda\sim\mathbb{P}_{\lambda},
\end{equation}
that measure the portion of the tissue under the oxygen threshold 0.1, which we take as the value under which hypoxia takes place. Our interest will be to estimate the probability density function of each $Q_{\lambda}$ and to provide a robust approximation of their expected value $\mathbb{E}\left[Q_{\lambda}\right]$. While these tasks can be achieved using classical Monte Carlo, the computational cost is enormous as it implies solving equation \eqref{eq:oxy} repeatedly. To alleviate this burden, we will replace the original PDE solver with a suitable MINN architecture trained to learn the map $\Lambda\to u_{\Lambda}$.

\subsection{Discretization and implementation details}
For the random generation of vascular networks we exploit Voronoi diagrams \cite{voronoi}. Let $\mathcal{P}:=\{P\subset\Omega\;|\;P\text{ finite}\}$ be the collection of all points tuples in $\Omega$. To any $P\in\mathcal{P}$, we associate the vascular graph $\Lambda(P)$ defined by the edges of the Voronoi cells generated by $P$. In this way, we obtain a correspondence $\mathcal{P}\to\Theta$ given by $P\to\Lambda(P)$, that we can exploit to prescribe probability measures over $\Theta$. To this end, let $\lambda>0$ and let $X_{\lambda}$ be a Poisson point process over $\Omega$ having a uniform intensity of $10\lambda$. We denote by $\tilde{\mathbb{P}}_{\lambda}$ the probability measure induced by $X_{\lambda}$ over $\mathcal{P}$. Then, we define $\mathbb{P}_{\lambda}:=\#\tilde{\mathbb{P}}_{\lambda}$ as the push-forward measure obtained via the action $P\to\Lambda(P)$. This ensures the wished behavior: higher values of $\lambda$ tend to generate more points in the domain and, consequently, denser graphs.

We now proceed to discretize the variational problem. As a first step, we note that the vascular graph $\Lambda$ is not given in terms of a parametrization, which makes it harder to compute integrals of the form $\int_{\Lambda}v(\s)d\s.$ As an alternative, we consider the smoothed approximation below,
\begin{equation}\label{eq:approx}\int_{\Lambda}v(\s)d\s\approx\int_{\Omega}v(\x)\phi_{\Lambda}(\x)d\x,\end{equation}
where 
$$\phi_{\Lambda}:=\frac{1}{\epsilon^{2}}\max\left\{\epsilon-\text{dist}(\x,\Lambda),0\right\}.$$
Here, $\epsilon>0$ is a smoothness parameter that we fix to $\epsilon=0.05$. It is not hard to prove that, for each $v\in\mathcal{C}(\overline{\Omega})$ fixed, the right-hand side of equation \eqref{eq:approx} converges to the left hand-side as $\epsilon\to 0$: \newstuff{for a detailed proof, the reader may refer to Appendix B}. 
Then, our operator of interest becomes
$\operator: \phi_{\Lambda}\to u_{\Lambda},$
and we can proceed with our usual discretization via P1 Finite Elements. To this end, we discretize the domain using a triangular mesh of stepsize $h=0.03$, which results in $N_{h}=7253$ degrees of freedom. Then, we allow $\lambda$ to vary uniformely in [1,10] and we generate a total of 4500 training snapshots accordingly to the probability distributions introduced previously. We exploit these snapshots to train the MINN model below,
$$V_{h}\xrightarrow{\;r=0.1\;} V_{3h}\to V_{3h}\xrightarrow{\;r=0.1\;}V_{h},$$
where the architecture has been defined in analogy to the one employed for the nonlinear PDE in Section \ref{sec:experiments1}.1. The network is trained for a total of 50 epochs and using the same criteria presented in Section \ref{sec:experiments1}.

\subsection{Results}
Once trained, the Mesh-Informed Neural Network reported an average $L^{2}$-error of $4.99\%$, with errors below 10\% for 488 out of 500 test instances. We considered these results satisfactory and we proceeded to sample a total of 100'000 solutions using our DNN model. More precisely, we considered 100 equally spaced values of $\lambda$ in [1,10], and for each of those we sampled 1'000 independent solutions. From there, we obtained an i.i.d. sample of size 1'000 for each of the $Q_{\lambda_{i}}$, where $\lambda_{i}=\{1+i/11\}_{i=0}^{99}$. Results are in Figure 8.
\begin{figure}\label{fig:lowoxy}
\includegraphics[width=\textwidth]{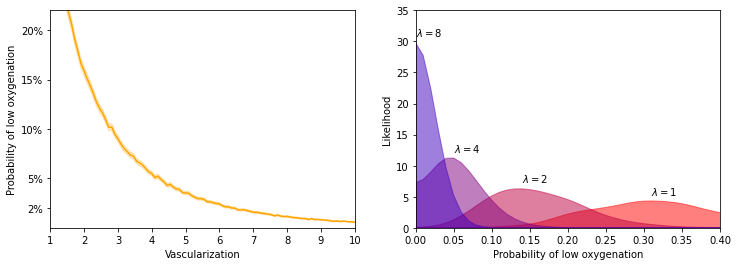}
\caption{Results for the UQ problem in Section \ref{sec:uq}. Left panel: expected probability of low oxygenation, $\mathbb{E}[Q_{\lambda}]$, as a function of the vascularization level $\lambda$. Confidence bands are computed pointwise using a 99\% confidence level. Right panel: probability distribution of $Q_{\lambda}$ for different values of $\lambda$. Colors fade from red to purple as $\lambda$ grows.}\end{figure}
\newline\newline
The left panel of Figure 8 shows the approximation of the map $\lambda\to\mathbb{E}[Q_{\lambda}]$. As the tightness of the 99\% confidence bands suggests, the estimate is rather robust. Spurious oscillations are most likely due to the numerical errors introduced by the MINN model, rather than from statistical noise. Coherently with the physical interpretation of $\lambda$, we see that the probability of low oxygenation decreases with the vascular density. Interestingly, although the total intensity of the source term is normalized to the same level in any configuration, the networks with higher gaps between neighboring edges are prone to spots of low oxygen concentration. Not only, the decay appears to be exponential. Further investigations seem to confirm this intuition, as we obtain an $R^{2}$-coefficient of 0.987 when trying to relate $\lambda$ and $\log\mathbb{E}[Q_{\lambda}]$ via linear regression. 
Conversely, the right panel of Figure 8 shows how the probability distribution of $Q_{\lambda}$ changes according to $\lambda$. The densities are more spread out when $\lambda$ is near 1, while they shrink towards zero as $\lambda$ increases. This is coherent with the physical intuition, and we would expect the density of $Q_{\lambda}$ to converge to a Dirac delta as $\lambda\to+\infty$.

In real scenarios where the physical complexity of a vascularized tissue is appropriately described as in \cite{possenti}, this analysis would be computationally viable only with the employment of the MINN model as a surrogate for the numerical solver. In the case presented here, both the full order model and the surrogate model are computationally inexpensive. However, the former required around 2 minutes to generate 1'000 PDE solutions. Conversely, the trained DNN model was able to provide the same number of solutions in as little as 3 milliseconds, corresponding to a speed up factor of approximately 40. For multiphysics models where a simulation of a single point in the parameter space could cost hours of wall computational time, such gain could enable approaches that would be otherwise unreasonable. Even in the present simplified setting, such a boost also makes up for the computational effort required to train the network. In fact: (i) collecting the training snapshots took 575.96 seconds, (ii) training the MINN model required 125.32 seconds, (iii) generating the 100'000 new solutions took 0.3 seconds. In contrast, the numerical solver would generate at most $\approx$ 5'500 solutions within the same amount of time. These considerations support the interest in further developing model order reduction techniques based on deep neural networks that are robust for general spatial domains, such as MINNs. In fact, we are currently developing model reduction techniques applied to realistic models of the vascular microenvironment that leverage on the DL-ROM framework, previously developed in \cite{franco,fresca,fresca2}, combined with the efficiency of MINNs.

\section{Conclusions}
\label{sec:conclusions}
In this paper, we have introduced Mesh-Informed Neural Networks (MINNs), a novel class of sparse DNN models that can be used to learn general operators between infinite dimensional spaces. The approach is based on an apriori pruning strategy that is obtained by embedding the hidden states into discrete functional spaces of different fidelities. Despite being very easy to implement, MINNs show remarkable advantages with respect to dense architectures, such as a massive reduction in the computational costs and an increased ability to generalize over unseen samples. This is coherent with the results available in the pruning literature \cite{pruning}, even though the setting differs from the one considered thereby.

We have tested MINNs over a large variety of scenarios, going from low dimensional manifolds to parameter dependent PDEs, showing that these architectures can \newstuff{be a competitive alternative for learning} nonlinear operators \newstuff{in the presence of complex spatial domains}. This opens a wide new range of research directions that we wish to investigate further in future works. For instance, one could test the use of MINNs in more sophisticated Deep Learning based Reduced Order Models for PDEs (DL-ROMs), such as those in \cite{franco,fresca,fresca2,lee}. \newstuff{In addition}, considering how MINNs are actually built, it \newstuff{would} be \newstuff{also} interesting to see whether one can take advantage of multi-fidelity strategies during the training phase, as in \cite{mengwu}. \newstuff{Finally, another intriguing research question is whether one can characterize the approximation properties of such architectures, similarly to what other researchers have already done for FNOs and DeepONets, see e.g. \cite{lanthaler2} and \cite{lanthaler}, respectively.}

\section*{Statements and Declarations}
\noindent\textbf{Funding}\\
This project has received funding under the ERA-NET ERA PerMed / FRRB grant agreement No ERAPERMED2018-244, RADprecise - Personalized radiotherapy: incorporating cellular response to irradiation in personalized treatment planning to minimize radiation toxicity.
\\\\
\newstuff{\textbf{Acknowledgments}\\
The authors would like to thank Prof. Jan S. Hesthaven, EPFL (École Polytechnique Fédérale de Lausanne), for his precious insights about the geodesic version of mesh-informed architectures.}
\\\\
\textbf{Data availability}\\
Enquiries about data availability should be directed to the authors.

\newpage
\appendix

\newstuff{\section*{Appendix A: Supplementary material for Section \ref{sec:experiments4}}}
\begin{figure}[H]
    \centering
    \begin{minipage}{0.49\textwidth}
    \includegraphics[width=\textwidth]{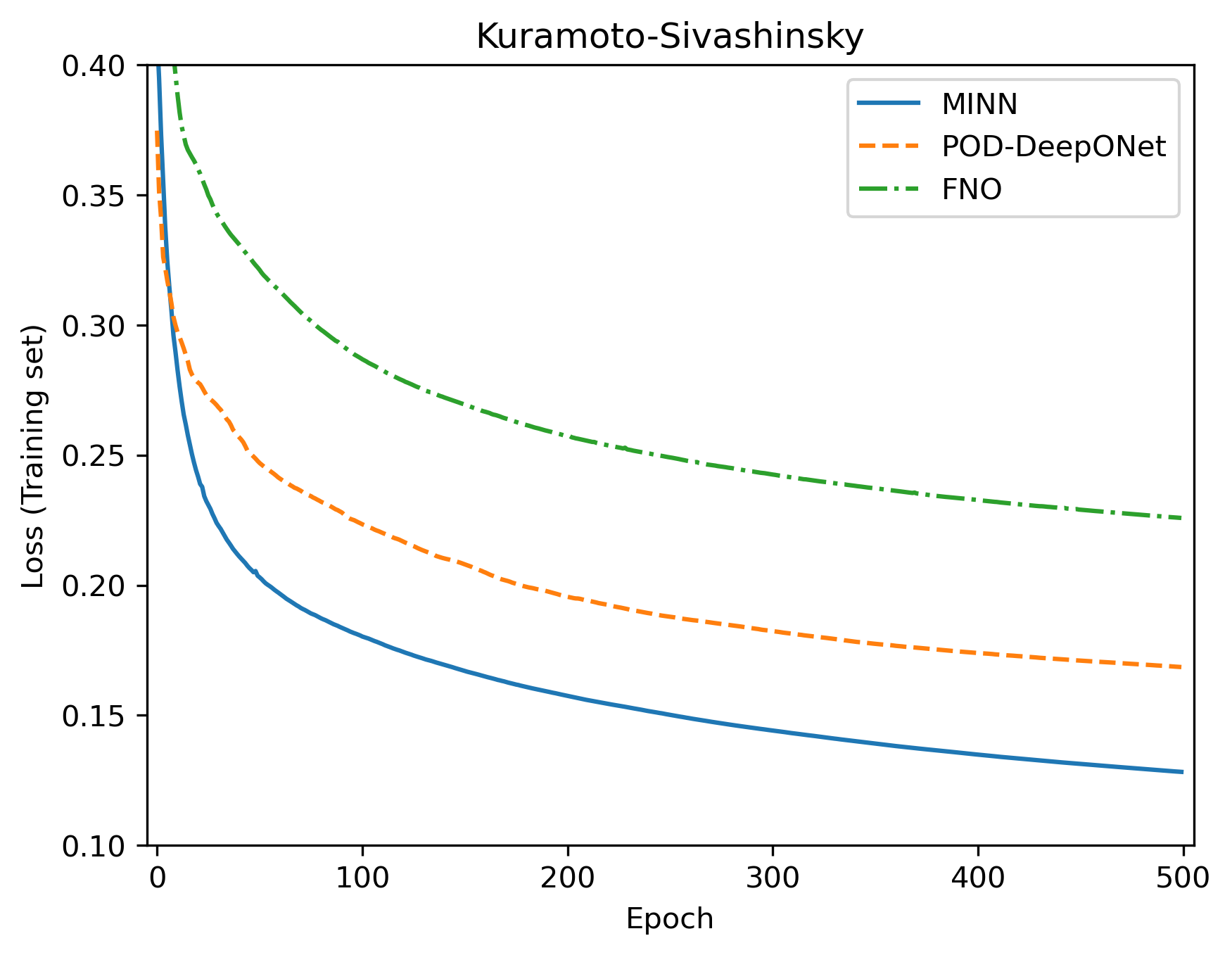}
    \end{minipage}
    \begin{minipage}{0.49\textwidth}
    \includegraphics[width=\textwidth]{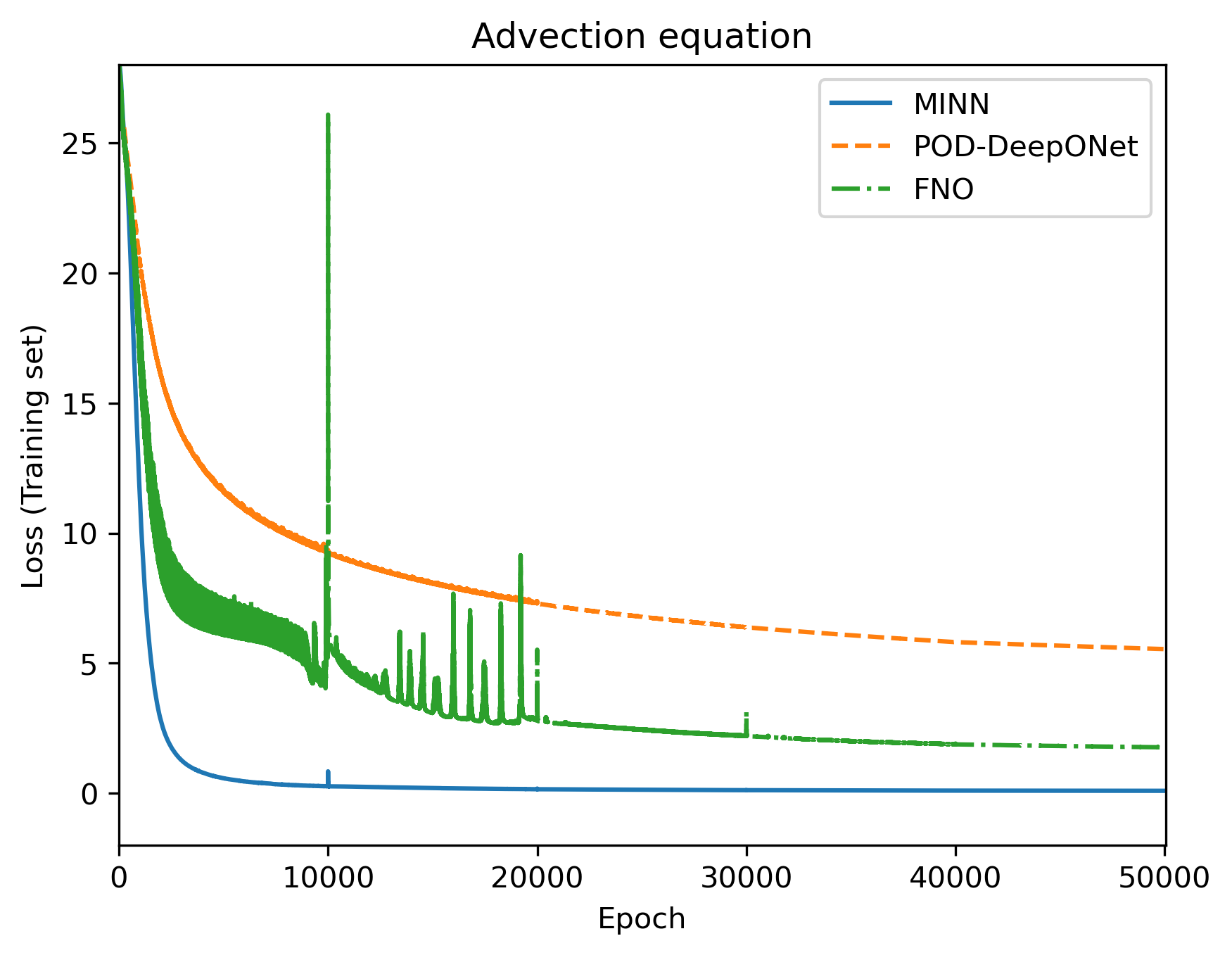}
    \end{minipage}
    \caption{\newstuff{Loss optimization for MINN, POD-DeepONet and FNO, for the two case studies reported in Sections \ref{subsec:kuramoto} and \ref{subsec:kolmogorov}. For the advection equation, only the training of the best POD-DeepONet is reported (50 basis).}}
    \label{fig:losses}
\end{figure}
\renewcommand{\arraystretch}{1.5}
\begin{table}[H]
    \centering
    \begin{tabular}{l | l l l}
    \hline
         \bf Case study & \bf MINN & \bf POD-DeepONet & \bf FNO \\
         \hline\hline
         Kuramoto-Sivashinsky §\ref{subsec:kuramoto} & 12m 7s & 3m 12s & 1h 33m 17s\\
         Advection equation §\ref{subsec:kolmogorov} & 5m 6s & 4m 7s & 6h 56m 29s\\\hline
    \end{tabular}
    \caption{\newstuff{Training times for MINN, POD-DeepONet and FNO, for the two case studies reported in Sections \ref{subsec:kuramoto} and \ref{subsec:kolmogorov}. For the advection equation, only the training time of the best POD-DeepONet is reported (50 basis).}}
    \label{tab:times}
\end{table}

\section*{Appendix B: Auxiliary result for Section \ref{sec:uq}}

\begin{lemma}
Let $\Omega$ be a Lipschitz domain. Let $\Lambda\subset\overline{\Omega}$ be the union of finitely many segments, where each segment intersects $\partial\Omega$ in at most two points (the extremes). For any fixed $v\in\mathcal{C}(\overline{\Omega})$ one has
$$\frac{1}{\epsilon^{2}}\int_{\Omega}v(\x)\max\left\{\epsilon-\textnormal{dist}(\x,\Lambda),0\right\}d\x\to \int_{\Lambda}v(\s)d\s\quad\text{as }\epsilon\downarrow0^{+}.$$
\end{lemma}
\begin{proof}
Let $\varphi_{\Lambda}^{\epsilon}$ be the (unscaled) kernel $$\varphi_{\Lambda}^{\epsilon}(\x):=\max\{\epsilon-\text{dist}(\x,\Lambda),0\}.$$ 
By definition, we note that $\varphi_{\Lambda}^{\epsilon}$ vanishes outside of the set $\Lambda+B(0,\epsilon):=\{\x+\epsilon\mathbf{v}\;|\;\x\in \Lambda,\;\mathbf{v}\in B(0,1)\}$. We now proceed in three steps.
\\\\
\textbf{Step 1.} We start by proving that the lemma holds whenever $\Lambda$ is composed by a single segment. Without loss of generality, we let $\Lambda = [0,1]\times\{0\}$. For the sake of simplicity, we further assume that $\Lambda\cap\partial\Omega=\emptyset$. The case in which $\Lambda$ has an extreme on the boundary can be handled similarly by exploiting the Lipschitz regularity of $\partial\Omega$. Let $\epsilon<\text{dist}(\Lambda,\partial\Omega)$, so that $\Lambda+B(0,\epsilon)\subset\Omega$. By direct computation we have
\begin{multline*}
\int_{\Omega}v(\x)\varphi_{\Lambda}^{\epsilon}(\x)d\x=\frac{1}{\epsilon^{2}}\int_{\Lambda+B(0,\epsilon)}v(\x)\varphi_{\Lambda}^{\epsilon}(\x)d\x=\\\\=\frac{1}{\epsilon^{2}}\int_{A_{\epsilon}\cup B_{\epsilon}}v(\x)\varphi_{\Lambda}^{\epsilon}(\x)d\x+\frac{1}{\epsilon^{2}}\int_{[0,1]\times[-\epsilon,\epsilon]}v(\x)\varphi_{\Lambda}^{\epsilon}(\x)d\x    
\end{multline*}
\noindent where $A_{\epsilon}$ and $B_{\epsilon}$ are two half circles of radius $\epsilon$ respectively centered at the extremes of the segment $\Lambda$. It is easy to see that the first contribute vanishes as $\epsilon\downarrow0^{+}$. In fact, since $||\varphi_{\Lambda}^{\epsilon}||_{\infty}=\epsilon$,
$$\left|\frac{1}{\epsilon^{2}}\int_{A_{\epsilon}\cup B_{\epsilon}}v(\x)\varphi_{\Lambda}^{\epsilon}(\x)d\x\right|\le\frac{1}{\epsilon^{2}}||v||_{\infty}\cdot\epsilon|A_{\epsilon}\cup B_{\epsilon}|=\pi\epsilon||v||_{\infty}.$$
Conversely, for the second term we have
$$\int_{0}^{1}\frac{1}{\epsilon^{2}}\int_{-\epsilon}^{\epsilon}v(x_{1},x_{2})\varphi_{\Lambda}^{\epsilon}(x_{1},x_{2})dx_{2}dx_{1}=\int_{0}^{1}\frac{1}{\epsilon^{2}}\int_{-\epsilon}^{\epsilon}v(x_{1},x_{2})(\epsilon-|x_{2}|)dx_{2}dx_{1}=$$
$$=\int_{0}^{1}\frac{1}{\epsilon^{2}}\int_{-1}^{1}v(x_{1},\epsilon z)(\epsilon-\epsilon|z|)\epsilon dzdx_{1}=\int_{0}^{1}\int_{-1}^{1}v(x_{1},\epsilon z)(1-|z|)dzdx_{1}.$$
By letting $\epsilon\downarrow0^{+}$ we then get 
$$\int_{0}^{1}\int_{-1}^{1}v(x_{1},0)(1-|z|)dzdx_{1}=\left(\int_{\Lambda}v(\s)d\s\right)\left(\int_{-1}^{1}(1-|z|)dz\right)=\int_{\Lambda}v(\s)d\s.$$

\noindent\\\\\textbf{Step 2.} Let $\Lambda=L_{1}\cup\dots\cup L_{n}$ be given by the union of $n$ segments. For any $i=1,\dots,n$, let $\hat{L}_{i}:=\{\x\in\Omega\;|\;\text{dist}(\x,L_{i})<\text{dist}(\x,\Lambda\setminus L_{i})\}$. We prove the following auxiliary result,
$$|(\Omega\setminus \hat{L}_{i})\cap(L_{i}+B(0,\epsilon))|=o(\epsilon^{2}).$$
To see this, we note that, upto sets of measure zero,
$$(\Omega\setminus \hat{L}_{i})\cap(L_{i}+B(0,\epsilon)) = (\hat{L}_{1}\cup\dots\hat{L}_{i-1}\cup\hat{L}_{i+1}\cup\dots \hat{L}_{n})\cap(L_{i}+B(0,\epsilon)).$$
It is then sufficient to prove that $|\hat{L_{j}}\cap(L_{i}+B(0,\epsilon))|=o(\epsilon^{2})$ for all $j$ independently. If $L_{i}\cap L_{j}=\emptyset$, the proof is trivial. Conversely, if the two segments intersect, let $\theta$ be the angle between the two lines. It is easy to see that the intersection $\hat{L_{j}}\cap(L_{i}+B(0,\epsilon))$ is contained in a triangle of height $\epsilon$ and width $\epsilon/\tan(\theta/2)+\epsilon\tan(\theta/2).$ The conclusion follows.
\\\\
\textbf{Step 3.} Let $\Lambda=L_{1}\cup\dots\cup L_{n}$ and define the regions $\hat{L}_{1},\dots\hat{L}_{n}$ as in the previous step. Fix any $v\in\mathcal{C}(\Omega)$. Then
$$\frac{1}{\epsilon^{2}}\int_{\Omega}v(\x)\varphi_{\Lambda}^{\epsilon}(\x)d\x = \frac{1}{\epsilon^{2}} \sum_{i=1}^{n}\int_{\hat{L}_{i}}v(\x)\varphi_{\Lambda}^{\epsilon}(\x)d\x=\sum_{i=1}^{n}\frac{1}{\epsilon^{2}}\int_{\hat{L}_{i}}v(\x)\varphi_{L_{i}}^{\epsilon}(\x)d\x.$$
Therefore, we can prove the original claim by showing that $\frac{1}{\epsilon^{2}}\int_{\hat{L}_{i}}v(\x)\varphi_{L_{i}}^{\epsilon}(\x)d\x\to\int_{L_{i}}v(\s)d\s$ for each $i$. At this purpose, fix any $i=1,\dots,n.$ We have
\begin{multline*}
\left|\frac{1}{\epsilon^{2}}\int_{\Omega}v(\x)\varphi_{L_{i}}^{\epsilon}(\x)d\x-\frac{1}{\epsilon^{2}}\int_{\hat{L}_{i}}v(\x)\varphi_{L_{i}}^{\epsilon}(\x)d\x\right|\le\frac{1}{\epsilon^{2}}\int_{\Omega\setminus \hat{L}_{i}}|v(\x)||\varphi_{L_{i}}^{\epsilon}(\x)|d\x\le\\\\
\le \frac{1}{\epsilon^{2}}||v||_{\infty}\cdot\epsilon|(\Omega\setminus \hat{L}_{i})\cap(L_{i}+B(0,\epsilon))| = o(\epsilon)
\end{multline*}
\noindent and the conclusion follows from Step 1.\qed
\end{proof}



\begin{thebibliography}{10}

\bibitem{rendering}
Akenine-Moller, T., Haines, E., and Hoffman, N. (2019). Real-time rendering. \textit{AK Peters/crc Press}.

\bibitem{augasta}
Augasta, M., and Kathirvalavakumar, T. (2013). Pruning algorithms of neural networks—a comparative study. \textit{Open Computer Science}, 3(3), 105-115.

\bibitem{voronoi} Aurenhammer, F. (1991). Voronoi diagrams—a survey of a fundamental geometric data structure. \textit{ACM Computing Surveys (CSUR)}, 23(3), 345-405.

\bibitem{stuart}
Bhattacharya, K., Hosseini, B., Kovachki, N.B., and Stuart, A.M. (2021). Model Reduction and Neural Networks for Parametric PDEs. \textit{SMAI Journal of
Computational Mathematics}, 7, 121-157.

\bibitem{pruning}  Blalock, D., Ortiz, J. J. G., Frankle, J., and Guttag, J. (2020). What is the state of neural network pruning?. \textit{Proceedings of Machine Learning and Systems 2020 (MLSys 2020)}.

\bibitem{brener} Berner, J., Grohs, P., Kutyniok, G., and Petersen, P. (2021). The modern mathematics of deep learning. \textit{arXiv preprint} arXiv:2105.04026.

\bibitem{lno2}
Cao, Q., Goswami, S., and Karniadakis, G. E. (2023). LNO: Laplace Neural Operator for Solving Differential Equations. \textit{arXiv preprint} arXiv:2303.10528.

\bibitem{cattaneo}
Cattaneo, L., Zunino, P. (2014). A computational model of drug delivery through microcirculation to compare different tumor treatments. \textit{International Journal for Numerical Methods in Biomedical Engineering}, 30 (11), pp. 1347-1371. 

\bibitem{lno1}
Chen, G., Liu, X., Li, Y., Meng, Q., and Chen, L. (2023). Laplace neural operator for complex geometries. \textit{arXiv preprint} arXiv:2302.08166.

\bibitem{ChenPINN} 
Chen, W., Wang, Q., Hesthaven, J.S., and Zhang, C. (2021).
Physics-informed machine learning for reduced-order modeling of nonlinear problems. \textit{Journal of Computational Physics}, 446, 110666.

\bibitem{devore}
Daubechies, I., DeVore, R., Foucart, S., Hanin, B., and Petrova, G. (2021). Nonlinear Approximation and (Deep) ReLU Networks. \textit{Constructive Approximation}, 1-46.

\bibitem{schwabhigh}  Elbrächter, D., Grohs, P., Jentzen, A., and Schwab, C. (2021). DNN expression rate analysis of high-dimensional PDEs: Application to option pricing. \textit{Constructive Approximation}, 1-69.

\bibitem{simplicial} Ern, A., Guermond, JL. (2021). Simplicial finite elements. In: Finite Elements I. \textit{Texts in Applied Mathematics}, vol 72. Springer, Cham.

\bibitem{evans}
Evans, L. C. (2022). \textit{Partial differential equations} (Vol. 19). American Mathematical Society.

\bibitem{meshnet}  Feng, Y., Feng, Y., You, H., Zhao, X., and Gao, Y. (2019, July). Meshnet: Mesh neural network for 3d shape representation. In \textit{Proceedings of the AAAI Conference on Artificial Intelligence} (Vol. 33, No. 01, pp. 8279-8286).

\bibitem{fisher} Fisher, R. A. (1937). The wave of advance of advantageous genes. \textit{Annals of eugenics}, 7(4), 355-369.

\bibitem{francocnn}
Franco, N. R., Fresca, S., Manzoni, A., and Zunino, P. (2023). Approximation bounds for convolutional neural networks in operator learning. \textit{Neural Networks}, 161, 129-141.

\bibitem{franco}  Franco, N. R., Manzoni, A., and Zunino, P. (2021). A Deep Learning approach to Reduced Order Modelling of Parameter Dependent Partial Differential Equations. \textit{arXiv preprint} arXiv:2103.06183.

\bibitem{fresca}  Fresca, S., Dede, L., and Manzoni, A. (2021). A comprehensive deep learning-based approach to reduced order modeling of nonlinear time-dependent parametrized PDEs. \textit{Journal of Scientific Computing}, 87(2), 1-36.

\bibitem{fresca2} Fresca, S., and Manzoni, A. (2022). POD-DL-ROM: enhancing deep learning-based reduced order models for nonlinear parametrized PDEs by proper orthogonal decomposition. \textit{Computer Methods in Applied Mechanics and Engineering}, 388, 114181.

\bibitem{Gao} Gao, H., Sun, L., and Wang, J.X (2021). PhyGeoNet: Physics-informed geometry-adaptive convolutional neural networks for solving parameterized steady-state PDEs on irregular domain, \textit{Journal of Computational Physics}, 428, 110079.

\bibitem{geist} Geist, M., Petersen, P., Raslan, M., Schneider, R., and Kutyniok, G. (2021). Numerical solution of the parametric diffusion equation by deep neural networks. \textit{Journal of Scientific Computing}, 88(1), 1-37.

\bibitem{marta}
Gladstone, R. J., Rahmani, H., Suryakumar, V., Meidani, H., D'Elia, M., and Zareei, A. (2023). GNN-based physics solver for time-independent PDEs. \textit{arXiv preprint} arXiv:2303.15681.

\bibitem{goswami}
Goswami, S., Kontolati, K., Shields, M. D., and Karniadakis, G. E. (2022). Deep transfer operator learning for partial differential equations under conditional shift. \textit{Nature Machine Intelligence}, 1-10.

\bibitem{gribonval} Gribonval, R., Kutyniok, G., Nielsen, M., and Voigtlaender, F. (2022). Approximation spaces of deep neural networks. \textit{Constructive Approximation}, 55(1), 259-367.

\bibitem{habbal}
Habbal, A., Barelli, H., and Malandain, G. (2014). Assessing the ability of the 2D Fisher–KPP equation to model cell-sheet wound closure. \textit{Mathematical Biosciences}, 252, 45-59.

\bibitem{gruber}  Gruber, A., Gunzburger, M., Ju, L., and Wang, Z. (2021). A Comparison of Neural Network Architectures for Data-Driven Reduced-Order Modeling. \textit{arXiv preprint} arXiv:2110.03442.

\bibitem{ingo}
Gühring, I., and Raslan, M. (2021). Approximation rates for neural networks with encodable weights in smoothness spaces. \textit{Neural Networks}, 134, 107-130.

\bibitem{mengwu}
Guo, M., Manzoni, A., Amendt, M., Conti, P., and Hesthaven, J. S. (2022). Multi-fidelity regression using artificial neural networks: efficient approximation of parameter-dependent output quantities. \textit{Computer methods in applied mechanics and engineering}, 389, 114378.

\bibitem{he}
He, K., Zhang, X., Ren, S., and Sun, J. (2015). Delving deep into rectifiers: Surpassing human-level performance on imagenet classification. In \textit{Proceedings of the IEEE international conference on computer vision} (pp. 1026-1034).

\bibitem{lanthaler2}
Kovachki, N., Lanthaler, S., and Mishra, S. (2021). On universal approximation and error bounds for fourier neural operators. \textit{The Journal of Machine Learning Research}, 22(1), 13237-13312.

\bibitem{kovachki}
Kovachki, N., Li, Z., Liu, B., Azizzadenesheli, K., Bhattacharya, K., Stuart, A., and Anandkumar, A. (2021). Neural operator: Learning maps between function spaces. \textit{arXiv preprint} arXiv:2108.08481.

\bibitem{vision}
Krizhevsky, A., Sutskever, I., and Hinton, G. (2012). Imagenet classification with deep convolutional neural networks, \textit{NIPS}.

\bibitem{kuramoto}
Kuramoto, Y. (1978). Diffusion-induced chaos in reaction systems. \textit{Progress of Theoretical Physics} Supplement, 64, 346-367.

\bibitem{kutyniok}  Kutyniok, G., Petersen, P., Raslan, M., and Schneider, R. (2021). A theoretical analysis of deep neural networks and parametric PDEs. \textit{Constructive Approximation}, 1-53.

\bibitem{lanthaler}
Lanthaler, S., Mishra, S., and Karniadakis, G. E. (2022). Error estimates for deeponets: A deep learning framework in infinite dimensions. \textit{Transactions of Mathematics and Its Applications}, 6(1), tnac001.

\bibitem{lee}  Lee, K., and Carlberg, K. T. (2020). Model reduction of dynamical systems on nonlinear manifolds using deep convolutional autoencoders. \textit{Journal of Computational Physics}, 404, 108973.

\bibitem{fnos} Li, Z., Kovachki, N., Azizzadenesheli, K., Liu, B., Bhattacharya, K., Stuart, A., and Anandkumar, A. (2020). Fourier neural operator for parametric partial differential equations. \textit{arXiv preprint} arXiv:2010.08895.

\bibitem{karniadakis}  Lu, L., Jin, P., Pang, G., Zhang, Z., and Karniadakis, G. E. (2021). Learning nonlinear operators via DeepONet based on the universal approximation theorem of operators. \textit{Nature Machine Intelligence}, 3(3), 218-229.

\bibitem{donfno}
Lu, L., Meng, X., Cai, S., Mao, Z., Goswami, S., Zhang, Z., and Karniadakis, G. E. (2022). A comprehensive and fair comparison of two neural operators (with practical extensions) based on fair data. \textit{Computer Methods in Applied Mechanics and Engineering}, 393, 114778.

\bibitem{hardy}  Park, Y. J. (2018). On the Continuity of the Hardy-Littlewood Maximal Function. \textit{Journal of the Chungcheong Mathematical Society}, 31(1), 43-46.

\bibitem{computervision} Perera, S., Barnes, N., He, X., Izadi, S., Kohli, P., and Glocker, B. (2015). Motion segmentation of truncated signed distance function based volumetric surfaces. \textit{In 2015 IEEE Winter Conference on Applications of Computer Vision} (pp. 1046-1053). IEEE.

\bibitem{possenti} Possenti, L., Cicchetti, A., Rosati, R., Cerroni, D., Costantino, M. L., Rancati, T., and Zunino, P. (2021). A Mesoscale Computational Model for Microvascular Oxygen Transfer. \textit{Annals of Biomedical Engineering}, 49(12), 3356-3373.

\bibitem{deepmind}  Pfaff, T., Fortunato, M., Sanchez-Gonzalez, A., and Battaglia, P. W. (2020). Learning mesh-based simulation with graph networks. \textit{arXiv preprint} arXiv:2010.03409.

\bibitem{ravasio}
Ravasio, A., Cheddadi, I., Chen, T., Pereira, T., Ong, H. T., Bertocchi, C., et al (2015). Gap geometry dictates epithelial closure efficiency. \textit{Nature communications}, 6(1), 7683.

\bibitem{scarselli}
Scarselli, F., Gori, M., Tsoi, A. C., Hagenbuchner, M., and Monfardini, G. (2008). The graph neural network model. \textit{IEEE transactions on neural networks}, 20(1), 61-80.

\bibitem{sherratt}
Sherratt, J. A., and Murray, J. D. (1990). Models of epidermal wound healing. \textit{Proceedings of the Royal Society of London. Series B: Biological Sciences}, 241(1300), 29-36.

\bibitem{paris} Wang, S., Wang, H., and Perdikaris, P. (2021). Learning the solution operator of parametric partial differential equations with physics-informed DeepONets. \textit{Science advances}, 7(40), eabi8605.

\bibitem{language}
Wu, Y., Schuster, M., Chen, Z., Le, Q., Norouzi, M., Macherey, W. et al. (2016). Google’s neural machine translation system: Bridging the gap between human and machine translation, \textit{arXiv preprint} arXiv:1609.08144.

\bibitem{meshcond}  Xu, J., Pradhan, A., and Duraisamy, K. (2021). Conditionally Parameterized, Discretization-Aware Neural Networks for Mesh-Based Modeling of Physical Systems. \textit{arXiv preprint} arXiv:2109.09510.

\bibitem{negri}
Quarteroni, A., Manzoni, A., and Negri, F. (2015). \textit{Reduced basis methods for partial differential equations: an introduction} (Vol. 92). Springer.

\bibitem{quarteroni} Quarteroni, A., and Valli, A. (2008). \textit{Numerical approximation of partial differential equations} (Vol. 23). Springer Science \& Business Media.


\end{thebibliography}
\end{document}